\documentclass[12pt]{article}
\usepackage{a4wide}
\usepackage{amsmath, amsthm, amsfonts, amssymb, bbm, bm}
\usepackage{graphics}
\usepackage{epstopdf}
\usepackage{xypic}
\usepackage[all]{xy}
\usepackage[english]{babel}
\usepackage[font=small,format=plain,labelfont=bf,up]{caption}
\usepackage{enumitem}
\usepackage{upgreek}
\usepackage{multirow}

\usepackage{accents}
\usepackage{scrextend}

\usepackage{hyperref}       
\hypersetup{       
   pdftitle={ On invertible 2-dimensional framed and r-spin topological field theories },
   colorlinks=true,        
   pdfstartview=FitH,      
   allcolors=[rgb]{0,0.5,0},        
	 bookmarksdepth=section, 
   bookmarksopen=false,     
   bookmarksnumbered=true 
}

\makeatletter
\newcommand{\sbullet}{
  \hbox{\fontfamily{lmr}\fontsize{.4\dimexpr(\f@size pt)}{0}\selectfont\textbullet}}

\makeatother
\usepackage{mathtools}

\usepackage{soul}

\usepackage{tabularx}
\usepackage{diagbox}

\usepackage{tikz-cd}
\usetikzlibrary{arrows}
\usepackage{extpfeil}
\usepackage{tikzit}

\tikzstyle{black dot}=[fill=black, draw=black, shape=circle, minimum size=3pt, inner sep=0pt]
\tikzstyle{black dot small}=[fill=black, draw=black, shape=circle, minimum size=3pt, inner sep=0pt]
\tikzstyle{big white circle}=[fill=white, draw=black, shape=circle, minimum width=0.75cm, inner sep=0pt]
\tikzstyle{white dot big}=[fill=white, draw=black, shape=circle, inner sep=1pt]
\tikzstyle{white dot}=[fill=white, draw=black, shape=circle, minimum size=3pt, inner sep=0pt]
\tikzstyle{flat box}=[fill=white, draw=black, shape=rectangle, minimum width=2.5cm, minimum height=0.5cm]
\tikzstyle{square}=[fill=white, draw=black, shape=rectangle]
\tikzstyle{semicircle}=[fill=red, draw=black, shape=semicircle]
\tikzstyle{flat box double}=[fill=white, draw=black, shape=rectangle, minimum width=1.5cm, minimum height=0.5cm]

\tikzstyle{mid arrow}=[-, postaction={on each segment={mid-arrow}}, line width=0.75pt]
\tikzstyle{end arrow}=[->, >=latex]
\tikzstyle{red mid arrow}=[-, draw={rgb,255: red,214; green,42; blue,51}, postaction={on each segment={mid arrow}}, line width=1pt]
\tikzstyle{blue}=[-, draw={rgb,255: red,23; green,37; blue,167}, line width=1pt, dashed]
\tikzstyle{marked edge}=[-, postaction={on each segment={red semicircle}}, line width=0.75pt]
\tikzstyle{thick line}=[-, line width=0.75pt]
\tikzstyle{thick dashed black line}=[-, line width=0.75pt, dashed]

\usetikzlibrary{decorations.pathreplacing,decorations.markings,shapes.geometric}
\tikzset{
  on each segment/.style={
    decorate,
    decoration={
      show path construction,
      moveto code={},
      lineto code={
        \path [#1]
        (\tikzinputsegmentfirst) -- (\tikzinputsegmentlast);
      },
      curveto code={
        \path [#1] (\tikzinputsegmentfirst)
        .. controls
        (\tikzinputsegmentsupporta) and (\tikzinputsegmentsupportb)
        ..
        (\tikzinputsegmentlast);
      },
      closepath code={
        \path [#1]
        (\tikzinputsegmentfirst) -- (\tikzinputsegmentlast);
      },
    },
  },
  mid-arrow/.style={postaction={decorate,decoration={
        markings,
        mark=at position .52 with {\arrow[#1]{stealth}}
      }}},
}
\tikzset{
  red semicircle/.style={postaction={decorate,decoration={
        markings,
        mark=at position .5 with {
		\arrow[#1]{Circle[left,fill=red,length=6pt,width=6pt]}
	}
      }}},
}
\graphicspath{figures/}

\DeclarePairedDelimiterX\setc[2]{\{}{\}}{\,#1 \;\delimsize\vert\; #2\,}

\newtheoremstyle{important-thm}
     {3pt}
     {3pt}
     {\slshape}
     {}
     {\bfseries}
     {.}
     {.5em}
     {}

\theoremstyle{plain}

\theoremstyle{important-thm}

\newtheorem{theorem}{Theorem}
\newtheorem{lemma}[theorem]{Lemma}
\newtheorem{proposition}[theorem]{Proposition}
\newtheorem{corollary}[theorem]{Corollary}
\theoremstyle{definition}
\newtheorem{remark}[theorem]{Remark}
\newtheorem{example}[theorem]{Example}
\newtheorem{definition}[theorem]{Definition}

 \numberwithin{equation}{section}
 \numberwithin{theorem}{section}

\DeclareMathOperator{\Ext}{Ext}
\DeclareMathOperator{\Hom}{Hom}

\DeclareMathOperator{\id}{id}

\renewcommand{\O}{\mathrm{O}}
\DeclareMathOperator{\SO}{SO}

\DeclareMathOperator{\Arf}{Arf}

\newcommand\Vect{{\mathcal{V}\hspace{-.5pt}ect}}

\newcommand\SVect{{\mathcal{S}\hspace{-1.5pt}\mathcal{V}\hspace{-.5pt}ect}}
\newcommand\SLine{{\mathcal{S}\hspace{-1.5pt}\mathcal{L}\hspace{-.5pt}ine}}
\renewcommand\Line{{\mathcal{L}\hspace{-.5pt}ine}}
\newcommand{\Bord}[2]{{\mathcal{B}\hspace{-.5pt}ord_{#1}^{\hspace{1pt} #2}}}
\newcommand{\SKK}[2]{{\operatorname{SKK}_{#1}^{#2}}}
\newcommand{\SKKp}[2]{{{\operatorname{SKK}'}_{#1}^{#2}}}
\newcommand{\Fun}{{\mathrm{Fun}_{\otimes}^\mathrm{inv}}}

\newcommand\spin[1]{{\mathrm{Spin}_2^{#1}}}

\newcommand{\norm}[1]{\left\|{#1}\right\|\hspace{-1.5pt}}

\newcommand{\Mod}[1]{\ \left(\text{mod}\ #1 \right)}

\newcommand\Cb            {\mathbb{C}}

\newcommand\Rb            {\mathbb{R}}
\newcommand\Zb            {\mathbb{Z}}

\newcommand\Cc            {\mathcal{C}}
\renewcommand\Dc            {\mathcal{D}}

\newcommand\Gc            {\mathcal{G}}

\newcommand\Tc            {\mathcal{T}}

\newcommand{\void}[1]{}

\newcommand\doi[2]        {\href{http://dx.doi.org/#1}{#2}}

\begin{document}

\thispagestyle{empty}
\def\thefootnote{\fnsymbol{footnote}}
\begin{flushright}
	{ZMP-HH/19-13}\\ 
	{Hamburger Beitr{\"a}ge zur Mathematik 796}
\end{flushright}
\vskip 3em
\begin{center}\LARGE
	On invertible 2-dimensional framed and $r$-spin topological field theories
\end{center}

\vskip 2em
\begin{center}
{\large 
L\'or\'ant Szegedy\footnote{Email: {\tt lorant.szegedy@univie.ac.at}}}
\\[1em]
{\small
Fachbereich Mathematik, Universit\"at Hamburg\\
Bundesstra\ss e 55, 20146 Hamburg, Germany
}
\footnote{Present address: Faculty of Physics, University of Vienna, Boltzmanngasse 5, 1090 Wien, Austria}
\end{center}

\vskip 2em

\begin{abstract}
	We classify invertible 2-dimensional framed and $r$-spin topological field theories
	by computing the homotopy groups and the $k$-invariant 
	of the corresponding bordism categories.
	The zeroth homotopy group of a bordism category is the usual Thom bordism group, 
	the first homotopy group can be identified with a Reinhart vector field bordism group,
	or the so called SKK group as observed by Ebert, B\"ockstedt--Svane and 
	Kreck--Stolz--Teichner.
	We present the computation of SKK groups for stable tangential structures.
	Then we consider non-stable examples: the 2-dimensional framed and $r$-spin SKK groups and compute them explicitly 
	using the combinatorial model of framed and $r$-spin surfaces of
	Novak, Runkel and the author.
\end{abstract}

\setcounter{footnote}{0}
\def\thefootnote{\arabic{footnote}}

\section{Introduction}

Invertible topological field theories (TFTs) have recently gained attention as
they are predicted to describe short range entangled topological phases of matter \cite{Kapustin:2014spt,Freed:2016inv}. 
The latter are defined as deformation classes of gapped Hamiltonians while the former are
symmetric monoidal functors from 
the category of $d$-dimensional bordisms with $G$-tangential structure $\Bord{d}{G}$ to 
the category of super vector spaces $\SVect$,
(or to any other symmetric monoidal category),
which land in the groupoid of super lines $\SLine$
(or respectively in the Picard subgroupoid $\Cc^{\times}$ of the target category $\Cc$).

A $G$-tangential structure on a $d$-dimensional manifold is given by a group homomorphism $G\to \O_d$ and a map from the manifold into $BG$ which lifts the classifying map of the tangent bundle.
A $G$-structure is called stable if it extends in an appropriate way to 
a $G'$-tangential structure in one dimension higher 
after stabilising the tangent bundle (Definition~\ref{def:stability-cond}).
In \cite{Freed:2016inv} invertible fully extended TFTs with stable tangential structures 
and with values in the category of super lines have been identified by maps of spectra
and these have been classified using the computational power of stable homotopy theory.

Non-extended invertible TFTs factor through the groupoid completion $||\Bord{d}{G}||$ of the 
bordism category. 
Hence the classification of non-extended invertible TFTs can be formulated by
understanding Picard groupoids and functors between them.
By a result of \cite{Sinh:1975pic,Joyal:1993bmc,Johnson:2012one}, Picard groupoids are classified by their 
zeroth and first homotopy groups and their $k$-invariant, 
i.e.\ two abelian groups and a particular group homomorphism between them.
Functors of Picard groupoids
are classified by group homomorphisms between the homotopy groups
which commute with the $k$-invariants and 
an Ext group (Theorem~\ref{thm:picard-grpd-classification}).

The zeroth homotopy group $\pi_0$ of $||\Bord{d}{G}||$
is the usual Thom bordism group $\Omega_{d-1}^{G}$.
In \cite[Sect.\,2.5]{Ebert:2009mb} it was observed that if one considers the embedded oriented bordism category, where bordisms are embedded in $\Rb^{\infty}$ (see e.g.\ \cite{Galatius:2009cob}), 
the first homotopy group $\pi_1$ can be identified with 
the Reinhart vector field bordism group \cite{Reinhart:1963vfc}.
An earlier result \cite[Thm.\,4.4]{Karras:1973SKKbook} shows that the latter is isomorphic to the
so called SKK group $\SKK{d}{\SO}$.
The group $\SKK{d}{G}$ is defined as the group completion of closed $d$-dimensional manifolds 
with $G$-structures with disjoint union as addition modulo the four-term SKK relation
(see Definition~\ref{def:SKK-group} and Appendix~\ref{app:defs-of-SKK}).
In \cite{Boekstedt:2012bord} it was shown that the first homotopy group for the embedded bordism category for arbitrary $G$-tangential structure is in fact $\SKK{d}{G}$ (they call the SKK relation the ``chimera relation'').
More recently Kreck-Stolz-Teichner noticed that this result also holds for the abstract bordism category, i.e.\ where the bordisms are not embedded in $\Rb^{\infty}$.
We present their proof of this result (Theorem~\ref{thm:SKK}) with their permission in Appendix~\ref{app:pf-thm:SKK}.
We also mention the related results on invertible oriented TFTs and SKK invariants in \cite{Rovi:2018skk}.

In case the $G$-tangential structure extends to another $G'$-tangential structure in one dimension higher
there exists a surjective group homomorphism 
$\SKK{d}{G}\twoheadrightarrow \Omega_{d}^{G'}$ to the bordism group
(Proposition~\ref{prop:SKK-bord-surj}).
In case the tangential structure is stable, by which we mean that it extends 
one dimension higher in a particular way (Definition~\ref{def:stability-cond})
the kernel of this map was computed by Kreck-Stolz-Teichner (Theorem~\ref{thm:SKK-ses} and Appendix~\ref{app:SKK-computation-stable}).

If the tangential structure in question is not stable, e.g.\ framings in dimension not equal to $1$ or $3$,
it should be possible to compute SKK groups by computing homotopy groups of Madsen-Tillmann spectra, but this is outside of the scope of this work.

In the present work we consider 2-dimensional invertible TFTs with framings and $r$-spin structures. 
The latter are tangential structures corresponding to the $r$-fold cover of $\SO_2$ 
which are not stable unless $r=1$ (which correspond to orientations) or $r=2$. 
Our main result in Theorem~\ref{thm:pi0-pi1-k-bord-r} lists the corresponding bordism groups and SKK groups explicitly\footnote{
	Some of these results appear in \cite{Randal:2011pic}, although our computations
	seem to be less technical.
} and it is proven using
the SKK relations and the combinatorial model of framed and $r$-spin surfaces 
of \cite{Novak:2015phd,Runkel:2018rs}.
Our computation follows a different strategy from that of Theorem~\ref{thm:SKK-ses}: We compute the kernel of the Euler characteristic via the SKK relations.

For a group $H$ let ${}_n H\subset H$ denote its $n$-torsion subgroup of order $n$ elements and let us introduce the shorthand $H/2=H/2H$.
Theorem~\ref{thm:pi0-pi1-k-bord-r} provides a full classification of invertible 2-dimensional framed and $r$-spin TFTs (Theorem~\ref{thm:inv-r-spin-tft-classification}): 
\begin{theorem}
	Isomorphism classes of invertible framed and $r$-spin TFTs with values in
	$\Cc$ is given by
	\begin{center}
		\begin{tabular}{c c|c}
		 \multicolumn{2}{c|}{$G$-structure} &$\pi_0(\Fun(\Bord{2}{G},\Cc))$\\
			\hline
			\multicolumn{2}{c|}{framing} & ${}_2(\pi_0(\Cc^{\times}))\hphantom{\times\pi_1(\Cc^{\times})}\times \pi_1(\Cc^{\times})/2$\\
			\multirow{2}{*}{$r$-spin}&($r$ even) & ${}_2(\pi_0(\Cc^{\times}))\times\pi_1(\Cc^{\times})\times \pi_1(\Cc^{\times})/2$\\
		      &($r$ odd) & $\hphantom{{}_2(\pi_0(\Cc^{\times}))\times} \pi_1(\Cc^{\times})\hphantom{\times\pi_1(\Cc^{\times})/2}$\\
		\end{tabular}
	\end{center}
\end{theorem}

\begin{corollary}
	In the case of $\Cc=\SVect$ we have
	\begin{center}
		\begin{tabular}{c c|c}
		 \multicolumn{2}{c|}{$G$-structure} &$\pi_0(\Fun(\Bord{2}{G},\SVect))$\\
			\hline
			\multicolumn{2}{c|}{framing} & $\Zb/2\hphantom{\times\Cb^{\times}}$\\
			\multirow{2}{*}{$r$-spin}&($r$ even) & $\Zb/2\times\Cb^{\times}$\\
		      &($r$ odd) & $\hphantom{\Zb/2\times}\Cb^{\times}$\\
		\end{tabular}
	\end{center}
	The group $\Zb/2$ is generated by the TFT $Z_\mathrm{Arf}$ computing the Arf invariant (Theorem~\ref{thm:Arf}),
	$a\in\Cb^{\times}$ correspond to the TFT $Z_{\chi}^a$ computing the Euler characteristic (Lemma~\ref{lem:eulerchar-SKK}).
\end{corollary}

The rest of the paper is organised as follows. 
In Section~\ref{sec:inv-tft} we review the notion of invertible TFTs with tangential structures, 
the classification of Picard groupoids and functors of them.
We introduce the SKK groups, and finally we present their computation in the stable case. 
In Section~\ref{sec:r-spin} we turn to dimension 2 and after a brief recollection of notions
on framed and $r$-spin surfaces we compute the corresponding bordism and SKK groups. 
This yields the classification of 2-dimensional framed and $r$-spin TFTs with arbitrary target.
Appendix~\ref{app:stable-computations} contain the proofs for the computation of the SKK groups in the stable case. 
These results were communicated to the author by Matthias Kreck, Stephan Stolz and Peter Teichner, we present the results with their kind permission.
In Appendix~\ref{app:proof} we give the proof of Lemma~\ref{lem:torus-rel} which is central 
for our computation of the framed and $r$-spin SKK groups in dimension 2.

\subsection*{Acknowledgments}
The author would like to thank Manuel Ara\'ujo, Bertram Arnold, Mark Penney, Oscar Randal-Williams, Ingo Runkel and Peter Teichner 
for helpful discussions and comments. 
The author is indebted to Matthias Kreck, Stephan Stolz and Peter Teichner for allowing them to present their results in this paper.
Finally, the author wishes to express their gratitude for the anonymous referee and the editor for providing invaluable suggestions to improve the first draft of this paper.
The author gratefully acknowledges the Max Planck Institute for Mathematics for hospitality and financial support. 
The author was partially funded by the DFG Research Training Group 1670 
``Mathematics Inspired by String Theory and Quantum Field Theory''.

\section{Invertible topological field theories}\label{sec:inv-tft}
In this section we review the notion of invertible topological field theories with tangential structure. 
Then we turn to the classification for which we need to compute bordism groups and the so called SKK groups.

Let $d\ge1$ be an integer, $G$ a topological group and $\xi:G\to \O_d$ 
a continuous group homomorphism. A \textsl{$G$-tangential structure (or $G$-structure)} on a
$d$-dimensional manifold $M$ is a homotopy class of maps $\varphi:M\to BG$ such that 
the diagram
\begin{equation}
	\begin{tikzcd}[column sep = large]
		&BG \ar{d}{B\xi}\\
		M\ar{ru}{\varphi}\ar{r}[swap]{c_{TM}}&B\O_d
	\end{tikzcd}
	\label{eq:tangential-structure}
\end{equation}
commutes up to homotopy, where $c_{TM}$ is the classifying map of the tangent bundle of $M$. 
Similarly, we define a $G$-structure on a $(d-1)$-dimensional manifold $S$ as a $G$-structure on $S\times\Rb$. 
We call a manifold with $G$-structure a \textsl{$G$-manifold}.

Consider two $(d-1)$-dimensional closed $G$-manifolds $(S_0,\varphi_0)$ and $(S_1,\varphi_1)$.
A \textsl{$d$-dimensional $G$-bordism} $(M,\varphi):(S_0,\varphi_0)\to (S_1,\varphi_1)$ is a 
compact $d$-dimensional $G$-manifold together with $G$-structure preserving embeddings $\iota_i:S_i\to M$ (boundary parametrisation)  which identify the disjoint union of $S_0$ and $S_1$ with the boundary of $M$.

The \textsl{category of $d$-dimensional $G$-bordisms} $\Bord{d}{G}$ has objects 
closed $(d-1)$-dimensional $G$-manifolds
and morphisms $G$-structure preserving diffeomorphism classes of $d$-dimensional $G$-bordisms. 
Composition of morphisms $\circ$ is defined by gluing along boundary parametrisation.
For more details on this definition see e.g.\ \cite{Turaev:2010hqft,Stolz:2011susy}.
The category $\Bord{d}{G}$ is symmetric monoidal with the disjoint union as tensor product.

Let $\Cc=(\Cc,\otimes,1_{\Cc},\sigma)$ be a symmetric monoidal category
with tensor product $\otimes$, tensor unit $1_{\Cc}$ and symmetry $\sigma$. 
A $d$-dimensional \textsl{topological field theory with $G$-structure (TFT)} 
is a symmetric monoidal functor $Z:\Bord{d}{G}\to\Cc$ \cite{Atiyah:1988tft,Segal:1988cft,Segal:1988mf}. 
\begin{definition}
  A \textsl{Picard groupoid} is a symmetric monoidal groupoid in which every object has an inverse with respect to the tensor product.
  The \textsl{Picard subgroupoid} $\Cc^{\times}\subseteq\Cc$ is the full subgroupoid of invertible objects.
  \label{def:picard-grpd}
\end{definition}
\begin{definition}
  A TFT $Z:\Bord{d}{G}\to\Cc$ is \textsl{invertible} if its image lies in
the Picard subgroupoid $\Cc^{\times}\subseteq\Cc$.
  \label{def:inv-tft}
\end{definition}
We write $\Fun(\Bord{d}{G},\Cc)$ for the category of invertible TFTs.

Denote with $||\Cc||$ the groupoid completion of $\Cc$ (for details see Definition~\ref{def:grpd-compl}). 
We have an essentially surjective functor $\Cc\to||\Cc||$.
\begin{proposition}
	\label{prop:invertible-TFT-factorises}
	For an invertible TFT $Z:\Bord{d}{G}\to\Cc$ there is a unique 
	symmetric monoidal functor $\tilde{Z}:||\Bord{d}{G}||\to\Cc^{\times}$ so that
	\begin{equation}
		\begin{tikzcd}[column sep = large]
			\Bord{d}{G}\ar{r}{Z}\ar{d}{}\ar{dr}{Z}&\Cc\\
			\|\Bord{d}{G}\|\ar[dashed]{r}[swap]{\tilde{Z}}&\Cc^{\times}\ar[right hook->]{u}{}
		\end{tikzcd}
		\label{eq:invertible-TFT-factorises}
	\end{equation}
	commutes.
\end{proposition}

If $\Cc$ has duals then $||\Cc||$ is in fact a Picard groupoid: a dual object is also an inverse.
Conversely, for an object $X$ in a Picard groupoid
its inverse $X^{-1}$ is its dual $X^{\vee}$. 
We write $\mathrm{ev}_{X}:X\otimes X^{-1}\to 1_{\Cc}$ and 
$\mathrm{coev}_{X}:1_{\Cc}\to X^{-1}\otimes X$ for the evaluation and the coevaluation morphisms.

	Let $\Cc$ be a Picard groupoid.
	The \textsl{zeroth homotopy group of $\Cc$} is 
	the abelian group $\pi_0(\Cc)$ of isomorphism classes of objects.
	The \textsl{first homotopy group of $\Cc$} is 
	the abelian group $\pi_1(\Cc)$ of automorphisms of the tensor unit of $\Cc$.
	The \textsl{$k$-invariant of $\Cc$} is the group homomorphism
	$k_{\Cc}:\pi_0(\Cc)\otimes\Zb/2\to\pi_1(\Cc)$ given by
	$k_{\Cc}(X):=\mathrm{ev}_{X}\circ\sigma_{X^{-1},X}\circ\mathrm{coev}_{X}$.

\begin{theorem}[{\cite{Sinh:1975pic,Joyal:1993bmc,Johnson:2012one}}]
	\label{thm:picard-grpd-classification}
	\begin{enumerate}
		\item 
	Picard groupoids are classified by  the zeroth and first homotopy groups $\pi_0$ and $\pi_1$ 
	and the $k$-invariant $k:\pi_0\otimes\Zb/2\to\pi_1$.

		\item 
	The set of isomorphism classes of 
	functors $\Cc\to\Dc$ of Picard groupoids
	is in bijection with the set of triples $(f_0,f_1,\alpha)$, 
	where $f_i\in\Hom(\pi_i(\Cc),\pi_i(\Dc))$  $i=0,1$ are group homomorphisms,
	which make the diagram
	\begin{equation}
		\begin{tikzcd}
			\pi_0(\Cc)\otimes\Zb/2\ar{r}{f_0}\ar{d}[swap]{k_{\Cc}}& 
			\pi_0(\Dc)\otimes\Zb/2\ar{d}{k_{\Dc}}\\
			\pi_1(\Cc)\ar{r}{f_1}&\pi_1(\Dc)
		\end{tikzcd}
		\label{eq:functor-of-Pic-grpds}
	\end{equation}
	commute
	and $\alpha\in\Ext(\pi_0(\Cc),\pi_1(\Dc))$.
	\end{enumerate}
\end{theorem}
We note that the triple
$(f_0,f_1,\alpha=0)$ determines a strict symmetric monoidal functor.
The different choices of $\alpha\in\Ext(\pi_0(\Cc),\pi_1(\Dc))$
parametrise different monoidal structures for the same underlying functor.
\begin{example}
	\label{ex:SLine}
	There are 2 symmetric braidings on the monoidal category of $\Zb/2$-graded vector spaces over $\Cb$,
	one with the usual flip map ($\Vect_{\Zb/2}$) and one which is $-1$ times the flip on purely odd components ($\SVect$).
	The corresponding Picard groupoids are
	\begin{itemize}
	  \item $\Vect_{\Zb/2}^{\times}=\Line_{\Zb/2}$ with $\pi_0=\Zb/2$, $\pi_1=\Cb^{\times}$, $k([\Cb^{0|1}])=+1$,
		\item $\SVect^{\times}=\SLine$ with $\pi_0=\Zb/2$, $\pi_1=\Cb^{\times}$, $k([\Cb^{0|1}])=-1$,
	\end{itemize}
	where $\pi_0$ is generated by the 1-dimensional odd vector space $\Cb^{0|1}$.
\end{example}
\begin{definition}
	\label{def:SKK-group}
	Let $S\in\Bord{d}{G}$, $M_i:S\to \emptyset$ and $N_i:\emptyset\to S$ for $i=1,2$
	be morphisms in $\Bord{d}{G}$.
	The \textsl{SKK group} $\SKK{d}{G}$ is the quotient of the group completion of the monoid of $G$-structure preserving diffeomorphism classes of closed 
	$d$-dimensional $G$-manifolds with disjoint union as product by
	the SKK relations which are of the form
	\begin{align}
		M_1\circ N_1 \sqcup M_2\circ N_2 \sim
		M_1\circ N_2 \sqcup M_2\circ N_1\ .
		\label{eq:def-SKK-rel}
	\end{align}
	We write $[M]\in\SKK{d}{G}$ for the class of the 
	closed $d$-dimensional $G$-manifold $M$.
\end{definition}
We present the proof of the following theorem in Appendix~\ref{app:pf-thm:SKK}.
\begin{theorem}
	\label{thm:SKK}
	For $||\Bord{d}{G}||$ the zeroth and first homotopy groups are given by
	the bordism group and the SKK group:
	\begin{align}
		\pi_0\left( ||\Bord{d}{G}|| \right)=\Omega_{d-1}^{G}
		\quad\text{and}\quad
		\pi_1(||\Bord{d}{G}||)=\SKK{d}{G}\ . 
		\label{eq:thm:SKK}
	\end{align}
\end{theorem}

\begin{corollary}
  Isomorphism classes of invertible TFTs with target $\SLine$ are 
  in bijection with SKK-invariants, i.e.\ with $\Hom(\SKK{d}{G},\Cb^{\times})$.
  This is because in \eqref{eq:functor-of-Pic-grpds}
  $f_1$ completely determines $f_0$,
  since the $k$-invariant of $\SLine$ is injective, 
  and that by the divisibility of $\Cb^{\times}$ the group $\Ext(\Omega_{d-1}^{G},\Cb^{\times})$ vanishes.
  \label{cor:sline-skk-inv}
\end{corollary}

	An important 
	SKK invariant is the Euler characteristic:
\begin{lemma} 
	\label{lem:eulerchar-SKK}
	The Euler characteristic is a group homomorphism:
	\begin{align}
		\chi:\SKK{d}{G}\to\Zb\ .
		\label{eq:lem:eulerchar-SKK}
	\end{align}
	Let $a\in\Cb^{\times}$.
	The corresponding invertible TFT
	\begin{align}
	  Z_\chi^a:\Bord{G}{d}\to \SVect\,,
	  \label{eq:euler-TFT}
	\end{align}
	sending every object to the tensor unit $\Cb$, evaluated on a closed $d$-dimensional $G$-manifold $M$ is
	\begin{align}
	  Z_\chi^a(M)=a^{\chi(M)}\,.
	  \label{eq:euler-TFT-value}
	\end{align}
\end{lemma}
\begin{proof}
  One checks that both sides of \eqref{eq:def-SKK-rel} have the same Euler characteristic using its elementary properties.
  On the left hand side we have
  \begin{align}
    \begin{aligned}
    &\chi( M_1\circ N_1 \sqcup M_2\circ N_2)=
    \chi( M_1\circ N_1)+\chi( M_2\circ N_2)\\
    =&
    \chi( M_1)+\chi(N_1)-\chi(S)+\chi( M_2)+\chi(N_2)-\chi(S)\ ,
    \end{aligned}
    \label{eq:lem:eulerchar-SKK-pf}
  \end{align}
  which is clearly the same as the right hand side $\chi(M_1\circ N_2 \sqcup M_2\circ N_1)$.

  By the previous corollary we need an SKK invariant valued in $\Cb^{\times}$: $[M]\mapsto a^{\chi(M)}$.
  The corresponding TFT is described in more detail in \cite{Quinnlectures}.
\end{proof}

\subsubsection*{Stable tangential structures}

Consider the situation for $\xi:G\to \O_d$ when there exists a topological group $G'$ containing $G$ as a subgroup and
a continuous group homomorphism $\xi':G'\to \O_{d+1}$ such that the following diagram commutes:
\begin{equation}
	\begin{tikzcd}
		G\ar[hookrightarrow]{r}{}\ar{d}[swap]{\xi}&G'\ar{d}{\xi'}\\
		\O_d\ar[hookrightarrow]{r}{}&\O_{d+1}\\
	\end{tikzcd}
	\label{eq:pre-stability-cond}
\end{equation}
Given a $G$-structure $\phi$ on $M$ one canonically obtains a $G'$-structure $\psi$ on $M$ by stabilising the tangent bundle and composing with the map induced by the inclusion $G\hookrightarrow G'$:
\begin{equation}
  \begin{tikzcd}
    &BG\ar[hook]{r}{}\ar{d}{B\xi}&BG'\ar{d}{B\xi'}\\
    M\ar{r}[swap]{c_{TM}}\ar{ru}{\phi}\ar[bend left= 90]{rru}{\psi}&B\O_d\ar[hook]{r}{}&B\O_{d+1}
  \end{tikzcd}
  \label{eq:GtoGp-str}
\end{equation}
This allows one to consider $(d+1)$-dimensional $G'$-bordisms and the corresponding bordism group $\Omega_d^{G'}$.

We give the proof of the following proposition in Appendix~\ref{app:SKK-bord}.
\begin{proposition}
	\label{prop:SKK-bord-surj}
	Let $G\subseteq G'$ and $\xi,\xi'$ be as in \eqref{eq:pre-stability-cond}
	There is a surjective group homomorphism
	\begin{align}
		\SKK{d}{G}\twoheadrightarrow\Omega_d^{G'}\ .
		\label{eq:prop:SKK-bord-surj}
	\end{align}
\end{proposition}

In the oriented and unoriented case the kernel of \eqref{eq:prop:SKK-bord-surj} is generated by $S^d$ \cite{Karras:1973SKKbook},
however for general $G$ the sphere may not have a $G$-structure. 
Requiring the tangential structure to be stable overcomes this difficulty 
and enables one to compute the kernel similarly as in the (un)oriented case.

\begin{definition}
  \label{def:stability-cond}
  We call tangential structures corresponding to
  $\xi:G\to \O_d$ \textsl{stable} if there exists $\xi:G'\to \O_{d+1}$ as in
  \eqref{eq:pre-stability-cond} such that the induced map  on the orbit space $G'/G\to \O_{d+1}/\O_d\cong S^d$  is a homeomorphism.
\end{definition}

Let us fix a stable tangential structure $\xi:G\to \O_d$ and a $d$-dimensional manifold $M$.
There is a bijection between isomorphism classes of $G$-structures on $M$ and isomorphism classes of
$G'$-structures over the stabilised tangent bundle of $M$ (or equivalently on $M\times\Rb$).
This can be seen by observing that defining a $G$-structure $\varphi$ for such a $G'$-structure $\psi$ is a lifting problem
\begin{equation}
  \begin{tikzcd}
    &S^d\ar{d}{}\ar{r}{\cong}&S^d\ar{d}{}\\
    &BG\ar{d}{}\ar{r}{B\xi}&B\O_d\ar{d}{}\\
    M\ar{r}[swap]{\psi}\ar{ru}{\varphi}&BG'\ar{r}{B\xi'}&B\O_{d+1}
  \end{tikzcd}
  \label{eq:GptoG-str-obs}
\end{equation}
The obstructions to such a lifting are elements in cohomology groups with coefficients in the low dimensional homotopy groups of the fiber $S^d$ of $BG\to BG'$ and hence all vanish.

For a stable $\xi:G\to \O_d$ one obtains a $G$-structure on $S^d$ by restricting the $G'$-structure on $\Rb^{d+1}$, which is unique up to isomorphism, as the normal bundle of $S^d\subset\Rb^{d+1}$ is trivial. 
We fix this $G$-structure on $S^d$ for the rest of this section.
Similarly one obtains a $G$-structure on any codimension 1 submanifold with trivial normal bundle of a $d$-dimensional $G$-manifold.

Recall the Kervaire semicharacteristic $s(m)=\sum_{i=0}^{2n}\dim(H^{2i}(M))\Mod{2}$ of a $(4n+1)$-dimensional manifold.
The proof of the following theorem is in Appendix~\ref{app:SKK-computation-stable}:

\begin{theorem}
	\label{thm:SKK-ses}
	Consider a stable tangential structure $\xi:G\to \O_d$ and the short exact sequence
	\begin{align}
		0\rightarrow K \hookrightarrow \SKK{d}{G}\twoheadrightarrow\Omega_d^{G'}\rightarrow 0\ .
		\label{eq:thm:SKK-ses}
	\end{align}
	\begin{enumerate}
	  \item The kernel $K$ is generated by $[S^d]$.
	    \label{part:1}
	  \item If $d=2n$ is even then
	    \begin{enumerate}[label=\alph*)]
	      \item $K\cong \Zb$, the isomorphism is given by the rescaled Euler characteristic $\chi/2$;
	    \label{part:2a}
	      \item \eqref{eq:thm:SKK-ses} splits if and only if $\chi(M)$ is even for every $d$-dimensional $G$-manifold $M$.
	    \label{part:2b}
	    \end{enumerate}
	    \label{part:2}
	  \item If $d=2n+1$ is odd then
	    \begin{enumerate}[label=\alph*)]
	      \item $K=0$ if there is a closed $d+1$-dimensional $G'$-manifold $W$ with $\chi(W)$ odd, and $K=\Zb/2$ otherwise;
	    \label{part:3a}
	      \item if $n$ is even and $\xi$ lands in $\SO_d$ (i.e.\ every $G$-manifold is oriented) then \eqref{eq:thm:SKK-ses} splits by the semicharacteristic.
	    \label{part:3b}
	    \end{enumerate}
	    \label{part:3}
	\end{enumerate}

\end{theorem}

\begin{remark}
\label{rem:framing-nonstable}
Note that if $G$ is the trivial group then the corresponding tangential structure is a framing which is only stable in dimension 1 or 3,
and $S^d$ has a framing only in dimensions $d=1$, 3 or 7.
Also, $r$-spin structures for $r>2$ are not stable.
\end{remark}

\section{Two-dimensional framed and \texorpdfstring{$r$}{r}-spin TFTs}\label{sec:r-spin}

In this section we introduce the notion of framed and $r$-spin surfaces and 
recall some properties of the respective bordism categories. 
We compute the corresponding SKK groups explicitly:
Instead of considering a surjective group homomorphism to a bordism group as in \eqref{eq:prop:SKK-bord-surj} and trying to determine its kernel,
we consider the (rescaled) Euler characteristic (which is an SKK invariant) and determine its kernel using a combinatorial model of framed and $r$-spin surfaces.
Finally we give the classification of invertible 2-dimensional framed and $r$-spin TFTs.

We start by sketching some definitions from \cite[Sec.\,2]{Runkel:2018rs}.
Let $r\in\Zb_{\ge0}$.
The \textsl{$r$-spin group $\spin{r}$} is the $r$-fold cover
for $r>0$ and the universal cover for $r=0$ of $\SO_2$.
We write $\xi:\spin{r}\to \SO_2$ for the covering map.
An \textsl{$r$-spin structure} on a surface $\Sigma$ is
the tangential structure on $\Sigma$ with respect to
$\xi:\spin{r}\to \SO_2\hookrightarrow \O_2$.

We will work with a skeletal version of the $r$-spin bordism category, 
which we also denote with $\Bord{2}{\spin{r}}$. It has objects 
$r$-spin circles, i.e.\ disjoint unions of pairs 
$(S^1,x)=S^1_x$, where $x\in\Zb/r$.
The morphisms are diffeomorphism classes of bordisms with $r$-spin structure.
Every $r$-spin bordism $\Sigma$ comes with two maps
$\kappa_\mathrm{in/out}:\pi_0(\partial_\mathrm{in/out}\Sigma)\to\Zb/r$
giving the types of the in- and outgoing boundary components.
(In \cite{Runkel:2018rs} the map $\lambda:\pi_0(\partial_\mathrm{in}\Sigma)\to\Zb/r$
is related to $\kappa_\mathrm{in}:\pi_0(\partial_\mathrm{in}\Sigma)\to\Zb/r$
via $\lambda=1-\kappa_\mathrm{in}$ and similarly 
the map $\mu:\pi_0(\partial_\mathrm{out}\Sigma)\to\Zb/r$
is related to $\kappa_\mathrm{out}$ via $\mu=1-\kappa_\mathrm{out}$.)

\begin{remark}
	\label{rem:zero-spin-framing}
	By \cite[Prop.\,2.2]{Runkel:2018rs} 0-spin structures correspond to framings,
	i.e. $\Bord{2}{\spin{0}}$ is equivalent to the framed 2-dimensional bordism category.
\end{remark}

\begin{proposition}[{\cite[Prop.\,2.21]{Runkel:2018rs}}]
	\label{prop:existence-of-r-spin-str}
	Let $\Sigma$ be a connected bordism of genus $g$ with 
	$b_\mathrm{in}$ ingoing and $b_\mathrm{out}$ outgoing 
	boundary components and
	$\kappa^\mathrm{in/out}:\pi_0(\partial\Sigma)^\mathrm{in/out}\to\Zb/r$.
	There exist $r$-spin structures on 
	$\Sigma$ if and only if
	\begin{align}
		\chi(\Sigma)=2-2g-b_\mathrm{in}-b_\mathrm{out}\equiv
		\sum_{j=1}^{b_\mathrm{out}}\kappa^\mathrm{out}(j)
		-\sum_{l=1}^{b_\mathrm{in}}\kappa^\mathrm{in}(l)
		\Mod{r}\ .
		\label{eq:prop:existence-of-r-spin-str}
	\end{align}
\end{proposition}
Let us write $\Sigma_{g,b}^\mathrm{in}$ (resp.\ $\Sigma_{g,b}^\mathrm{out}$) 
for a connected bordism of genus $g$ with 
$b$ ingoing (resp.\ outgoing) boundary components only and $\Sigma_g$ for $b=0$.
For $r$ even, tensoring the TFT of \cite[Thm.\,1.3]{Runkel:2018rs} with $Z_\chi^{2^{-1/2}}$ we get:
\begin{theorem}
	\label{thm:Arf}
	If $r$ is even then there is an invertible $r$-spin TFT $Z_\mathrm{Arf}$ which
	computes the Arf invariant. For $\Sigma_{g}$ with $g$ satisfying 
	\eqref{eq:prop:existence-of-r-spin-str} and an $r$-spin structure $\varphi$ on $\Sigma_{g}$ we have
	\begin{align}
		Z_{\Arf}(\Sigma_{g},\varphi)=
		(-1)^{\Arf(\varphi)}\ .
		\label{eq:thm:arf}
	\end{align}
\end{theorem}
\begin{example}
	\label{ex:torus-sigma-gb}
	\leavevmode
	\begin{enumerate}
		\item There exist $r$-spin structures on the sphere if and only if $r=1$ or $r=2$.
		\item There exist $r$-spin structures on the torus for every value of $r$. 
			The isomorphism classes of $r$-spin structures on a fixed torus
			are in bijection with $(\Zb/r)^2$ and we write
			$T(s,t)$ for an $r$-spin torus corresponding to $(s,t)\in(\Zb/r)^2$. 
			The mapping class group of the torus $SL(2,\Zb)$ acts on
			$(\Zb/r)^2$ via the standard action
			and the orbits, i.e.\ diffeomorphism classes
			of $r$-spin tori are in bijection with the divisors of $r$
			\cite{Geiges:2012rs,Runkel:2018rs}.
		\item Let $\tilde{r}:=r/\mathrm{gcd}(r,2)$.  
			There exist $r$-spin structures on $\Sigma_{g}$ 
			if and only if 
			\begin{align}
				g\equiv 1\Mod{\tilde{r}}\ .
				\label{eq:ex:torus-sigma-gb}
			\end{align}
			We will write $U_l=\Sigma_{1+l\tilde{r}}$.
			If $1+l\tilde{r}\ge2$, there is one mapping class group
			orbit of $r$-spin surfaces with underlying surface 
			$U_l$ for $r$ odd and two for $r$ even \cite{Geiges:2012rs,Randal:2014rs,Runkel:2018rs}.
			The latter two are distinguished by the Arf invariant
			and we denote these $r$-spin surfaces by
			$U_l^{(+)}$ (Arf invariant $+1$)
			and $U_l^{(-)}$ (Arf invariant $-1$).

		\item The disc $\Sigma_{0,1}$ has a unique $r$-spin structure up to isomorphism.
			In case the boundary is outgoing then it is of type $+1$,
			if it is ingoing then it is of type $-1$.
	\label{ex:torus-sigma-gb:4}
	\end{enumerate}
\end{example}

After this recollection of notions we turn to the computation of the group $\SKK{2}{\spin{r}}$. For this we will look at some particular SKK-relations.
The following lemma will be proved in Appendix~\ref{app:proof} using the
combinatorial model of $r$-spin surfaces of \cite{Runkel:2018rs}.
\begin{lemma}
	\label{lem:torus-rel}
	In $\SKK{2}{\spin{r}}$ we have the relations
	\begin{align}
		[T(\kappa,u_1)]+ [T(\kappa,u_2)]+ [T(\kappa,u_3)]+ [T(\kappa,u_4)]&=
		[T(\kappa,u_1+u_2)]+[T(\kappa,u_3+u_4)]\ ,
		\label{eq:lem:torus-rel:1}\\
		[T(1,u)]&=0\ ,
		\label{eq:lem:torus-rel:2}
	\end{align}
	for every $\kappa,u,u_i\in\Zb/r$ ($i=1,\dots,4$).
	The $k$-invariant of $S^1_{\kappa}$ is $T(\kappa,0)$.
\end{lemma}

\begin{lemma}
	\label{lem:subgroup-tori}
	The subgroup 
	$\Tc^{(r)} \subset \SKK{2}{\spin{r}}$ 
	generated by 
	$r$-spin tori is $\Zb/2$ for $r$ even and trivial for $r$ odd.
\end{lemma}
\begin{proof}
	By \eqref{eq:lem:torus-rel:1} in Lemma~\ref{lem:torus-rel} for
	$u_1=u_3=0$ we get
	\begin{align}
		2[T(\kappa,0)]=0
		\label{eq:pf::prop:pi0-pi1-k-bord-r:2torsion}
	\end{align}
	for every $\kappa\in\Zb/r$. Since for arbitrary $s,t\in\Zb/r$
	we have $[T(s,t)]$=$[T(\mathrm{gcd}(s,t),0)]$, every $r$-spin torus has 2-torsion in $\SKK{2}{\spin{r}}$.

	Again by \eqref{eq:lem:torus-rel:1}, now for $u_1=0$, $u_3=1$ and using
	\eqref{eq:lem:torus-rel:2} we get
	\begin{align}
		[T(\kappa,0)]+[T(\kappa,u)]=[T(\kappa,u+1)]\ .
		\label{eq:pf:prop:pi0-pi1-k-bord-r:reduce1}
	\end{align}
	Combining the latter with \eqref{eq:pf::prop:pi0-pi1-k-bord-r:2torsion} we get
	\begin{align}
		[T(\kappa,u)]=[T(\kappa,u+2)]\ .
		\label{eq:pf:prop:pi0-pi1-k-bord-r:reduce2}
	\end{align}
	Therefore every $r$-spin torus is equal to $[T(0,0)]$ or $[T(1,0)]$,
	the latter being zero by \eqref{eq:lem:torus-rel:2} 
	of Lemma~\ref{lem:torus-rel}.

	If $r$ is odd then $[T(0,0)]=[T(1,0)]$ and hence $\Tc^{(r)}=\{0\}$.

	For $r$ even Theorem~\ref{thm:Arf} gives an SKK-invariant
	$\SKK{2}{\spin{r}}\to\Cb^{\times}$ computing the 
	Arf invariant. It has value $-1\in \Cb^{\times}$ on $[T(0,0)]$, 
	which shows that in $\SKK{2}{\spin{r}}$ the element $[T(0,0)]$ 
	is non-trivial and hence $\Tc^{(r)}\cong\Zb/2$.

\end{proof}
The following lemma can be proven using the SKK-relations,
Lemma~\ref{lem:torus-rel}
and Theorem~\ref{thm:Arf}, i.e.\ the fact that the Arf invariant is compatible with glueing.
\begin{lemma}
	\label{lem:higher-genus-rel}
	Recall the $r$-spin surfaces $U_l$ ($r$ odd) and $U_l^{(\pm)}$ ($r$ even) from Example~\ref{ex:torus-sigma-gb}.
	We have
	\begin{align}
		[U_l^{(+)}]+[T(0,0)]&=[U_l^{(-)}]\ ,
		\label{eq:lem:higher-genus-rel:2}
	\end{align}
	\begin{align}
		[U_l]+[U_j]= [U_{l+j}]\quad\text{ and }\quad
		[U_l^{(+)}]+[U_j^{(+)}]= [U_{l+j}^{(+)}]\ ,
		\label{eq:lem:higher-genus-rel:1}
	\end{align}
	for every $l,j\in\Zb_{\ge0}$.
	If $r\le2$ we furthermore have
	\begin{align}
		[\Sigma_{g+1}]+[S^2]=[\Sigma_g]\quad\text{ and }\quad
		[\Sigma_{g+1}^{(+)}]+[S^2]=[\Sigma_g^{(+)}]\ ,
		\label{eqeq:lem:higher-genus-rel:3}
	\end{align}
	for every $g\in\Zb_{\ge0}$.
\end{lemma}

\begin{theorem}
	\label{thm:pi0-pi1-k-bord-r}
	The zeroth and first homotopy groups 
	and the $k$-invariant 
	of $||\Bord{2}{\spin{r}}||$ are the following:
	\begin{center}
	\begin{tabular}{c|c|c|c}
		$r$ & $\pi_0$ & $\pi_1$ & $k:\pi_0\to\pi_1$\\
		\hline
		$0$ & $\Zb/{2}$ & $\Zb/{2}$ & $\id$\\
		$>0$, even & $\Zb/{2}$ & $\Zb\times\Zb/{2}$ & $x\mapsto (0,x)$\\
		$>0$, odd & $\{0\}$ & $\Zb$ & $0$\\
	\end{tabular}
	\end{center}
\end{theorem}
\begin{proof}
	We start with $\pi_1$, which
	by Theorem~\ref{thm:SKK} is $\SKK{2}{\spin{r}}$.
	Note that the cases $r=1$ and $r=2$ could be treated using Theorem~\ref{thm:SKK-ses}. 
	Here we present a computation which applies for arbitrary values of $r$.

	By Lemma~\ref{lem:eulerchar-SKK} the Euler characteristic
	is a group homomorphism $\chi:\SKK{2}{\spin{r}}\to\Zb$.
	We claim that the kernel of $\chi$ is 
	$\Tc^{(r)}$
	the subgroup generated by $r$-spin tori. 
	This can be seen by observing that
	any element in $\SKK{2}{\spin{r}}$ can be brought to the form
	$[U_l^{(\varepsilon)}]-[U_{l'}^{(\varepsilon)}]$ (or $[U_l]-[U_{l'}]$)
	for $l,l'\in\Zb_{\ge1}$ and $\varepsilon\in\{\pm\}$
	if $r>2$, or to a multiple of $[S^2]$
	if $r\le2$ up to $r$-spin tori using Lemma~\ref{lem:higher-genus-rel}.

	If $r=0$ then $\chi=0$ (cf.\ Example~\ref{ex:torus-sigma-gb})
	and by Lemma~\ref{lem:subgroup-tori} we have $\SKK{2}{\spin{0}}=\Zb/2$.

	If $r>0$ then by Proposition~\ref{prop:existence-of-r-spin-str}
	the values of $\chi$ are divisible by $2\tilde{r}$, 
	where $\tilde{r}=r/\mathrm{gcd}(2,r)$
	and $\chi/(2\tilde{r})$ is surjective so
	we have a short exact sequence
	\begin{align}
		\Tc^{(r)}\xhookrightarrow{\qquad}
		\SKK{2}{\spin{r}}\xtwoheadrightarrow{\chi/(2\tilde{r})}\Zb\ .
		\label{eq:SKK-r-spin-ses}
	\end{align}

	We can define a section of $\chi/(2\tilde{r})$ as follows.
	Let $j=\varepsilon |j|\in\Zb$ and let us define
	\begin{align}
		\begin{aligned}
		\varphi:\Zb&\to\SKK{2}{\spin{r}}\\
		j&\mapsto 
		\begin{cases}
			j[S^2]&\text{; if $r=1,2$,}\\
			\varepsilon [U_{|j|}]&\text{; if $r>2$ is odd,}\\
			\varepsilon [U_{|j|}^{(+)}]&\text{; if $r>2$ is even.}\\
		\end{cases}
		\end{aligned}
		\label{eq:chi-section}
	\end{align}
	By Lemma~\ref{lem:higher-genus-rel} $\varphi$ is a group homomorphism
	and it is clearly a section of $\chi/(2\tilde{r})$, 
	so \eqref{eq:SKK-r-spin-ses} splits which together with 
	Lemma~\ref{lem:subgroup-tori} proves the first part of the theorem.

	\medskip
	We continue with computing $\pi_0=\Omega_1^{\spin{r}}$. 
	Let $\kappa_i\in\Zb/r$ be fixed for $i=1,\dots,n$.
	By \eqref{eq:prop:existence-of-r-spin-str} it is possible to choose $g$
	such that
	\begin{align}
		2-2g-(n+1)+\sum_{i=1}^{n}\kappa_i\equiv 
		\begin{cases}
			1\Mod{r}&\text{; if $r$ is odd,}\\
			0\text{ or }1\Mod{r}&\text{; if $r$ is even,}
		\end{cases}
		\label{eq:mod2kappa}
	\end{align}
	so that there exists an $r$-spin bordism
	\begin{align}
		\bigsqcup_{i=1}^{n}S^1_{\kappa_i}
		\to S^1_{\kappa}
		\label{eq:pi0-bordism}
	\end{align}
	with underlying surface $\Sigma_{g,n+1}$ with $n$ ingoing and 
	one outgoing boundary component and where $\kappa=1$ for $r$ odd 
	or $\kappa\in\left\{ 0,1 \right\}$ for $r$ even. 
	That is, in $\Omega_1^{\spin{r}}$ every element is equal to
	$[S^1_0]$ or to $[S^1_1]$.

	Recall that the disc gives an $r$-spin bordism $S^1_1\to\emptyset$,
	so if $r$ is odd then $\Omega_1^{\spin{r}}=\{0\}$ and
	if $r$ is even then $\Omega_1^{\spin{r}}$ is generated by $[S^1_0]$.
	
	It is easy to see from the previous discussion that $[S^1_0]$
	has 2-torsion. For the rest of the proof let us assume that $r$ is even.
	By Lemma~\ref{lem:torus-rel} the $k$-invariant of $[S^1_0]$ is 
	$[T(0,0)]\in\pi_1$, which is non-zero showing that $[S^1_0]$ is non-zero.
	Altogether we get that $\Omega_1^{\spin{r}}=\Zb/2$ for $r$ even.

\end{proof}

We are ready to prove our main result:

\begin{theorem}
	\label{thm:inv-r-spin-tft-classification}
	The group of isomorphism classes of invertible $r$-spin TFTs with values in
	$\Cc$ is given by
	\begin{center}
		\begin{tabular}{c|c}
			$r$&$\pi_0(\Fun(\Bord{2}{\spin{r}},\Cc))$\\
			\hline
			$0$ & ${}_2(\pi_0(\Cc^{\times}))\hphantom{\times\pi_1(\Cc^{\times})}\times \pi_1(\Cc^{\times})/2$\\
			$>0$, even & ${}_2(\pi_0(\Cc^{\times}))\times\pi_1(\Cc^{\times})\times \pi_1(\Cc^{\times})/2$\\
		      $>0$, odd & $\hphantom{{}_2(\pi_0(\Cc^{\times}))\times} \pi_1(\Cc^{\times})\hphantom{\times\pi_1(\Cc^{\times})/2}$\\
		\end{tabular}
	\end{center}
\end{theorem}
\begin{proof}
	We use Theorem~\ref{thm:picard-grpd-classification} to compute
	$\pi_0(\Fun(\Bord{2}{\spin{r}},\Cc))$.
	If $r=0$ then $f_0\in\Hom(\Zb/2,\pi_0(\Cc^{\times}))={}_2(\pi_0(\Cc^{\times}))$ determines $f_1$. 
	Furthermore $\Ext(\Zb/2,\pi_1(\Cc^{\times}))=\pi_1(\Cc^{\times})/2$.
	If $r>0$ is even then by \eqref{eq:functor-of-Pic-grpds}
	we can write $f_1:\Zb\times \Zb/2\to\pi_{1}(\Cc^{\times})$ 
	as $f_1(x,y)=k_{\Cc^{\times}}\circ f_0(y)+f_1(x,0)$.
	So $f_1$ is completely determined by 
	$f_0\in\Hom(\Zb/2,\pi_0(\Cc^{\times}))$ and by $f_1(1,0)\in\pi_1(\Cc^{\times})$. 
	If $r>0$ is odd then $f_0=0$ and $f_1\in \Hom(\Zb,\pi_1(\Cc^{\times}))\cong \pi_1(\Cc^{\times})$.
	In this case $\Ext(\left\{ 0 \right\},\pi_1(\Cc^{\times}))=\left\{ 0 \right\}$.
\end{proof}

\appendix

\section{SKK groups for stable tangential structure}
\label{app:stable-computations}
The results in this section, in particular the proof of Theorems~\ref{thm:SKK}~and~\ref{thm:SKK-ses}, were communicated to the author by Matthias Kreck, Stephan Stolz and Peter Teichner. We present their proof with their permission.

Throughout this section we will use the symbol $+$ for the disjoint union $\sqcup$ of manifolds for better readability.
\subsection{Equivalent definitions of the SKK groups}
\label{app:defs-of-SKK}

Recall the equivalence relation in Definition~\ref{def:SKK-group} defining the SKK groups.
Originally the SKK groups for unoriented and oriented manifolds were defined via a slightly different equivalence relation \cite{Karras:1973SKKbook}.
Let $S,S'$ be $(d-1)$-dimensional closed unoriented (oriented) manifolds, $M_i:S'\to \emptyset$ and $N_i:\emptyset\to S$ for $i=1,2$ $d$-dimensional unoriented (oriented) bordisms and $\phi,\psi: S\to S'$ (orientation preserving) diffeomorphisms.
We define the composition
\begin{align}
  M_1\cup_{\phi} N_1:&=
  M_1\circ C_{\phi}\circ N_1\,, 
  \label{eq:def-diffeo-comp}
\end{align}
where $C_{\phi}$ is the mapping cylinder of $\phi$
\begin{align}
  C_{\phi}&=(S\times[0,1]\sqcup S)/( (s,1)\sim {\phi}(s))\quad\text{for $s\in S$,}
\end{align}
and the equivalence relation $\sim'$ on $d$-dimensional closed unoriented (oriented) manifolds as
\begin{align}
  M_1\cup_{\phi} N_1 + M_2\cup_{\psi}  N_2 \sim'
  M_1\cup_{\psi}  N_1 + M_2\cup_{\phi} N_2\ .
  \label{eq:def-SKK-rel-2}
\end{align}
Denote the quotients of the group completion of $\Bord{d}{(\SO)}(\emptyset,\emptyset)$ 
by equivalence relations $\sim$ and $\sim'$ by $\SKK{d}{(\SO)}$ and $\SKKp{d}{(\SO)}$ respectively.

\begin{proposition}
  The groups $\SKK{d}{(\SO)}$ and $\SKKp{d}{(\SO)}$ are isomorphic.
  \label{prop:def-SKK-2-rel}
\end{proposition}
\begin{proof}
  We show that two manifolds are equivalent under $\sim$ if and only if they are equivalent under $\sim'$.

  Consider $M_1 \cup_{\phi} N_1+M_2\cup_{\psi} N_2$
  and define the bordisms 
  \begin{align}
    \bar{M}_1:=M_1\circ C_{\phi}:S\to\emptyset\,,\quad\bar{M}_2:=M_2\circ C_{\psi}:S\to\emptyset\,.
    \label{eq:bar-M}
  \end{align}
  Then
  \begin{align}
    M_1 \cup_{\phi} N_1+M_2\cup_{\psi} N_2=
    \bar{M}_1 \circ N_1+\bar{M}_2\circ N_2\ .
    \label{eq:sim-1}
  \end{align}
  We apply the relation $\sim$ and obtain
  \begin{align}
    \begin{aligned}
    \bar{M}_1 \circ N_1+\bar{M}_2\circ N_2&\sim
    \bar{M}_1 \circ N_2+\bar{M}_2\circ N_1=\\
    M_1 \cup_{\phi} N_2+M_2\cup_{\psi} N_1&=
    N_2^* \cup_{\phi^{-1}} M_1^*+N_1^*\cup_{\psi^{-1}} M_2^*\ ,
    \end{aligned}
    \label{eq:sim-2}
  \end{align}
  where $M_i^*$ and $N_i^*$ denote the transposed bordisms.
  We proceed as above and compose with mapping cylinders to get
  \begin{align}
    N_2^* \cup_{\phi^{-1}} M_1^*+N_1^*\cup_{\psi^{-1}} M_2^*=
    \bar{N}_2^* \circ M_1^*+\bar{N}_1^*\circ M_2^*
    \label{eq:sim-3}
  \end{align}
  and apply the relation $\sim$ again:
  \begin{align}
    \begin{aligned}
    \bar{N}_2^* \circ M_1^*+\bar{N}_1^*\circ M_2^*&\sim
    \bar{N}_2^* \circ M_2^*+\bar{N}_1^*\circ M_1^*=\\
    N_2^* \cup_{\phi^{-1}} M_2^*+N_1^*\cup_{\psi^{-1}} M_1^*&=
    M_2 \cup_{\phi} N_2+M_1\cup_{\psi} N_1\ .
    \end{aligned}
    \label{eq:sim-4}
  \end{align}
  We see that relation $\sim$ implies relation $\sim'$.

  For the converse direction let us assume $S=S'$ and let $\tau:=\sigma_{S,S}$ denote the braiding.
  By applying relation $\sim'$ we have
  \begin{align}
    \begin{aligned}
    M_1 \circ N_1+M_2\circ N_2+(M_1+M_1)\cup_{\tau}(N_1+N_1)&=\\
    (M_1 + M_2)\cup_{\id_S} (N_1+N_2)+(M_1+M_1)\cup_{\tau}(N_1+N_1)&\sim'\\
    (M_1 + M_2)\cup_{\tau} (N_1+N_2)+(M_1+M_1)\cup_{\id_S}(N_1+N_1)&=\\
    M_1 \circ N_2+M_2\circ N_1+(M_1+M_1)\cup_{\id_S}(N_1+N_1)&\ .
    \end{aligned}
    \label{eq:simp-1}
  \end{align}
  Clearly $(M_1+M_1)\cup_{\tau}(N_1+N_1)$ and $(M_1+M_1)\cup_{\id_S}(N_1+N_1)$ are the same morphisms, so we have
  \begin{align}
    M_1 \circ N_1+M_2\circ N_2 \sim' M_1 \circ N_2+M_2\circ N_1\ .
    \label{eq:simp-2}
  \end{align}
  So relation $\sim'$ implies relation $\sim$.

\end{proof}

\subsection{Proof of Theorem~\ref{thm:SKK}}
\label{app:pf-thm:SKK}

\begin{definition}
  A \textsl{groupoid completion of a category $\Cc$} is a groupoid $\norm{\Cc}$ and a functor $\Cc\to\norm{\Cc}$ satisfying the following universal property.
  For every groupoid $\Gc$ and functor $F:\Cc\to\Gc$ there exists a unique functor $\tilde{F}:\norm{\Cc}\to\Gc$ such that
  the following diagram commutes:
  \begin{equation}
    \begin{tikzcd}
      \Cc\ar{rr}{}\ar{dr}[swap]{F}&&\norm{\Cc}\ar{dl}{\tilde{F}}\\
      &\Gc&
    \end{tikzcd}\,.
    \label{eq:grpd-compl-univ}
  \end{equation}
  \label{def:grpd-compl}
\end{definition}

A standard construction of $\norm{\Cc}$ is as follows (see e.g.\ \cite[Sec.\,1]{Gabriel:1967cfht}).
The objects of $\norm{\Cc}$ are the same as the objects of $\Cc$.
A morphism $f:X\to Y$ in $\norm{\Cc}$ is an equivalence class of sequences of objects $X_i$ and morphisms
$f_i$ either $X_{i-1}\to X_i$ or $X_{i-1}\leftarrow X_{i}$ in $\Cc$ for $i=1,\dots,n$ with $X_0=X$ and $X_n=Y$.
The equivalence relation is generated by the following moves:
\begin{enumerate}
  \item Two morphisms $X_{i-1}\xrightarrow{f_i} X_i\xrightarrow{f_{i+1}} X_{i+1}$ 
    ($X_{i-1}\xleftarrow{f_i} X_i\xleftarrow{f_{i+1}} X_{i+1}$)
    can be replaced by the morphism $X_{i-1}\xrightarrow{f_{i+1}\circ f_i} X_{i+1}$
    (resp.\ $X_{i-1}\xleftarrow{f_{i}\circ f_{i+1}} X_{i+1}$).
    \label{move:1}
  \item If $X_{i-1}=X_{i+1}$ and $f_i=f_{i+1}$ with $X_{i-1}\xrightarrow{f_i}X_i\xleftarrow{f_{i+1}}X_{i+1}$, 
    (or $X_{i-1}\xleftarrow{f_i}X_i\xrightarrow{f_{i+1}}X_{i+1}$)
    they can be omitted from the sequence.
    \label{move:2}
\end{enumerate}
Composition is concatenation of sequences.
The functor $\Cc\to\norm{\Cc}$ is identity on objects and sends a morphism $X\xrightarrow{f}Y$ to the sequence
$X=X_0\xrightarrow{f_1=f}X_1=Y$.
Note that by Move~\ref{move:2} the inverse in $\norm{\Cc}$ of (the image of) a morphism $X\xrightarrow{f} Y$ in $\Cc$ is
$X_0=Y\xleftarrow{f} X=X_1$.

There is a group homomorphism
\begin{align}
  \text{group completion of }\Cc(X,X)\to \norm{\Cc}(X,X)
  \label{eq:lem:grp-completion:1}
\end{align}
induced by the functor $\Cc\to\norm{\Cc}$ which may not be surjective.
However under a mild condition it is the case (see also \cite[Prop.\,3.2]{Juer:2013inv}):

\begin{lemma}
  Let $X\in\Cc$ be such that for every $Y\in\Cc$ the set of morphisms $\Cc(X,Y)$ is non-empty if and only if $\Cc(Y,X)$ is non-empty.
  \begin{enumerate}
    \item 
      The group homomorphism \eqref{eq:lem:grp-completion:1} is surjective.
  \label{lem:grp-completion:1}
    \item 
      The quotient of the group completion of $\Cc(X,X)$ by the equivalence relation
      \begin{align}
	(g_1\circ f_2)^{-1}(g_1\circ f_1)\sim
	(g_2\circ f_2)^{-1}(g_2\circ f_1)
	\label{eq:lem:grp-completion:2}
      \end{align}
      for $f_i\in\Cc(X,Y)$, $g_i\in\Cc(Y,X)$
      is isomorphic to $\norm{\Cc}(X,X)$.
  \label{lem:grp-completion:2}
  \end{enumerate}
  \label{lem:grp-completion}
\end{lemma}
\begin{proof}
  \textsl{Part~\ref{lem:grp-completion:1}:}
  Let $f_i$ ($i=1,\dots,n$) be a sequence of morphisms representing $f\in\norm{\Cc}$.
  Let us assume that $f_1$ is a morphisms $X=X_0\to X_1$ (which we can by inserting the identity morphism if necessary).
  Using Moves~\ref{move:1}~and~\ref{move:2} we can assume that the $f_i$'s alternate directions:
  \begin{equation}
    \begin{tikzcd}
      f=\Big(\dots\ar{r}{f_6}&X_5&\ar{l}[swap]{f_5}X_4
      \ar{r}{f_4}&X_3&\ar{l}[swap]{f_3}X_2
    \ar{r}{f_2}&X_1&\ar{l}[swap]{f_1}X\Big)\ .
    \end{tikzcd}
    \label{eq:alternating-arrows}
  \end{equation}
  By the condition on $\Cc$ we obtain morphisms $h_i$ inductively:
  \begin{equation}
    \begin{tikzcd}
      &&X\ar{d}[swap]{h_4}
      &&X\ar{d}[swap]{h_2}&&\\
      \dots\ar{r}{f_6}&X_5\ar{d}[swap]{h_5}&\ar{l}[swap]{f_5}X_4
      \ar{r}{f_4}&X_3\ar{d}[swap]{h_3}&\ar{l}[swap]{f_3}X_2
      \ar{r}{f_2}&X_1\ar{d}[swap]{h_1}&\ar{l}[swap]{f_1}X\ ,\\
      &X&&X&&X&
    \end{tikzcd}
    \label{eq:alternating-arrows-h}
  \end{equation}
  e.g.\ by the existence of $X_1\xleftarrow{f_1}X$ there exists $X_1\xrightarrow{h_1}X$,
  then for $X_2\xrightarrow{f_2}X_1\xrightarrow{h_1}X$ there exists $X_2\xleftarrow{h_2}X$, etc.
  Using Moves~\ref{move:1}~and~\ref{move:2} one checks that 
  \begin{align}
    f=\dots(h_4\circ f_5\circ h_5) (h_3\circ f_4\circ h_4)^{-1}
    (h_2\circ f_3\circ h_3) (h_1\circ f_2\circ h_2)^{-1}
    (h_1\circ f_1)
    \label{eq:f-in-gr-CXX}
  \end{align}
  which is a composition of morphisms in the group completion of $\Cc(X,X)$.

  \textsl{Part~\ref{lem:grp-completion:2}:}
  Consider the diagram for the morphisms in \eqref{eq:lem:grp-completion:2}:
  \begin{equation}
    \begin{tikzcd}
      X\ar{r}{f_1}&Y\ar[bend right]{d}[swap]{g_1}\ar[bend left]{d}{g_2}&X\ar{l}[swap]{f_2}\\
      &X&
    \end{tikzcd}
    \label{eq:eq-rel-diag}
  \end{equation}
  Observe that by Move~\ref{move:2} both sides of \eqref{eq:lem:grp-completion:2} map to the same morphism in $\norm{\Cc}$ under \eqref{eq:lem:grp-completion:1}, so we have a well defined surjective map from the quotient.

  In order to show that this map is an isomorphism we define a map in the opposite direction.
  Let $f\in\norm{\Cc}(X,X)$ with a sequence of morphisms $f_i$ as in \eqref{eq:alternating-arrows}
  and choose morphisms $h_i$ as in \eqref{eq:alternating-arrows-h}.
  As above, we get a morphism $\bar{f}$ in the group completion.
  For another choice of morphisms $h_i'$ in  \eqref{eq:alternating-arrows-h} we get a morphism $\bar{f}'$.
  By looking at the diagram in \eqref{eq:eq-rel-diag} we see that $\bar{f}=\bar{f}'$ in the quotient,
  so the map into the quotient is well defined.

\end{proof}

We note that the conditions of the lemma are satisfied for any symmetric monoidal category with duals.
Now we are ready to prove Theorem~\ref{thm:SKK}.

\begin{proof}[Proof of Theorem~\ref{thm:SKK}]
  We apply Lemma~\ref{lem:grp-completion} Part~\ref{lem:grp-completion:2} for $\Cc=\Bord{d}{G}$ (which is a symmetric monoidal category with duals) and $X=\emptyset$.
  Observe that the relation in \eqref{eq:lem:grp-completion:2} is exactly the defining relation \eqref{eq:def-SKK-rel} of the SKK group.
\end{proof}

\subsection{Proof of Proposition~\ref{prop:SKK-bord-surj}}
\label{app:SKK-bord}
\begin{proof}[Proof of \ref{prop:SKK-bord-surj}]
  We start with a general construction.
  Let $Z\in\Bord{d}{G}$, $Z\xrightarrow{f}Z$ a diffeomorphism of $G$-manifolds and
  $\emptyset\xrightarrow{Y}Z$, $Z\xrightarrow{X}\emptyset$ $G$-bordisms.
  Let $1_X$ be the $G'$-bordism (with corners) with underlying $G'$-manifold $X\times [0,1]$ and define
  $1_Y$ and $1_Z$ similarly.
  We define a $G'$-bordism (with corners) $W_f:1_Z\to C_f$ to the mapping cylinder:
  \begin{align}
    W_f:=\left( Z\times[0,1]\times[0,1] \sqcup Z\right)/( (z,1,1) \sim f(z))
    \label{eq:def-Wf}
  \end{align}
  with the obvious boundary parametrisation.
  We define the $G'$-bordism $W$ as 
  \begin{equation}
    \begin{tikzcd}
      &&{}&&\\
      \emptyset&\ar{l}[swap]{X}Z&&Z\ar[bend right = 90]{ll}[swap]{C_f}\ar[bend left = 90]{ll}{1_Z}&\emptyset\ar{l}[swap]{Y}\\
      &&{}\ar[Rightarrow]{uu}{W_f}&&\\
    \end{tikzcd}
    \label{eq:def-W}
  \end{equation}
  by gluing $1_X$, $W_f$ and $1_Y$ along the last coordinate (i.e.\ as the horizontal composition in an extended bordism 2-category).
  
  We need to show that there is a $G'$-bordism between the two sides of \eqref{eq:def-SKK-rel}.
  Let $X=M_1+M_2$, $Y:=N_1+ N_2$, $Z:=S+S$,
  and $\tau:=\sigma_{S,S}:Z\to Z$ the flip map.
  The required bordism is
  \begin{equation}
  M_1\circ N_1 +M_2\circ N_2 =
  (M_1+M_2)\circ (N_1+N_2)
  \xrightarrow{W}
  (M_1+M_2)\circ \tau\circ (N_1+N_2)=
  M_1\circ N_2 +M_2\circ N_1 
  \ .
    \label{eq:SKK-rel-bord}
  \end{equation}
\end{proof}
\subsection{Computation of SKK groups in the stable case}
\label{app:SKK-computation-stable}

\begin{proposition}
  Let $\xi:G\to \O_d$ be stable, $M,N$ closed $d$-dimensional $G$-manifolds and $W:M\to N$ a $(d+1)$-dimensional $G'$-bordism.
  Then in $\SKK{d}{G}$ we have
  \begin{align}
    [M]-[N]=(\chi(M)-\chi(W))[S^d]
    \label{eq:prop:ker-pi}
  \end{align}
  \label{prop:ker-pi}
\end{proposition}
\begin{proof}
  We proceed in two steps:  First we show that if $W$ is the trace of a single surgery the proposition holds.
  Then we write $W$ as a composition of bordisms, each of which is the trace of a single surgery and apply the above result for each bordism.

  Assume that $W$ is obtained by attaching a $(k+1)$-handle $D^{k+1}\times D^{d-k}$ to $M\times [0,1]$, so $N$ is obtained by $(k+1)$-surgery on $M$.
  We can decompose $M$ and $N$ into $d$-dimensional $G$-bordisms
  \begin{align}
    \begin{aligned}
    M&=
    \left(\emptyset\xleftarrow{M\setminus\mathrm{Inn}(S^k\times D^{d-k})}S^k\times S^{d-k-1}\xleftarrow{S^k\times D^{d-k}}\emptyset\right)\\
    N&=
    \left(\emptyset\xleftarrow{M\setminus\mathrm{Inn}(S^k\times D^{d-k})}S^k\times S^{d-k-1}\xleftarrow{D^{k+1}\times S^{d-k-1}}\emptyset\right)
    \end{aligned}
    \label{eq:dec-M-N-surgery}
  \end{align}
  where the $G$-structures are induced by inclusion of submanifolds with trivial normal bundles of the handle.
  Note that since every (connected component of a) handle is contractible, there is a unique up to isomorphism $G'$ structure on it.
  Furthermore consider the Heegaard decomposition 
  \begin{align}
    S^d= \left(\emptyset\xleftarrow{S^k\times D^{d-k}}S^k\times S^{d-k-1}\xleftarrow{D^{k+1}\times S^{d-k-1}}\emptyset\right)\ .
    \label{eq:heegaard-S-d}
  \end{align}

  In $\SKK{d}{G}$ we have
  \begin{align}
    \begin{aligned}
      [M]+[S^d]&=[M\setminus\mathrm{Inn}(S^k\times D^{d-k})\circ S^k\times D^{d-k}] + [S^k\times D^{d-k}\circ D^{k+1}\times S^{d-k-1}]\\
    &=[M\setminus\mathrm{Inn}(S^k\times D^{d-k})\circ 
    D^{k+1}\times S^{d-k-1}]
    + [S^k\times D^{d-k}\circ 
    S^k\times D^{d-k}]\\
    &= [N]+[S^k \times S^{d-k}]
    \end{aligned}
    \label{eq:M-N-S-d}
  \end{align}
  and for $M=S^d$ we have
  \begin{align}
    [S^d]+[S^d]=[S^{k+1}\times S^{d-k-1}]+[S^k\times S^{d-k}]\ .
    \label{eq:S-d-S-d}
  \end{align}
  For $k=0$ we get $[S^1\times S^{d-1}]=0$,
  and inductively for $k>0$:
  \begin{align}
    [S^k\times S^{d-k}]=(1+(-1)^k)[S^d]=
    \begin{cases}
      2[S^d] &\text{ ($k$ even),}\\
     0 &\text{ ($k$ odd).}
    \end{cases}
    \label{eq:k-cases}
  \end{align}
  From this we obtain
  \begin{align}
    [M]=[N]+(-1)^k [S^d]\ .
    \label{eq:M-N-S-d-simple}
  \end{align}

  Using elementary properties of the Euler characteristic we compute
  \begin{align}
    \begin{aligned}
    \chi(W)&=\chi(M\times [0,1])+\chi(D^{k+1}\times D^{d-k})-\chi(S^k\times D^{d-k})\\
  &=\chi(M)\chi([0,1])+\chi(D^{k+1}) \chi(D^{d-k})-\chi(S^k)\chi( D^{d-k})\\
  &=\chi(M)+1-(1+(-1)^k)=\chi(M)-(-1)^{k}\ .
    \end{aligned}
    \label{eq:chi-W-M}
  \end{align}
  From this we directly get \eqref{eq:prop:ker-pi}:
  \begin{align}
    [M]=[N]+(\chi(M)-\chi(W))[S^d]\ .
    \label{eq:ker-pi-single}
  \end{align}

  \medskip

  Now consider the decomposition
  \begin{align}
    N\xleftarrow{W}M=\left(N=M_n\xleftarrow{W_n}\dots \xleftarrow{W_2}M_1\xleftarrow{W_1}M_0=M\right)\ ,
    \label{eq:dec-W-single}
  \end{align}
  where each $W_i$ is the trace of a single surgery.
  By \eqref{eq:ker-pi-single} we have for every $i=1,\dots,n$ that
  \begin{align}
    [M_{i-1}]=[M_i]+(\chi(M_{i-1})-\chi(W_i))[S^d]\ ,
    \label{eq:ker-pi-single-i}
  \end{align}
  and therefore
  \begin{align}
    [M]=[M_0]=[M_1]+(\chi(M_0)-\chi(W_1))[S^d]=\dots=[M_n]+\sum_{i=1}^{n}(\chi(M_{i-1})-\chi(W_{i}))[S^d]\ .
    \label{eq:ker-pi-multiple-1}
  \end{align}
  On the other hand we have
  \begin{align}
    \chi(W)=\chi(W_n)+\chi\left( \Pi_{i=1}^{n-1}W_i \right)-\chi(M_{n-1})=\dots=\sum_{i=1}^n \chi(W_{i})-\sum_{j=1}^{n-1}\chi(M_{j})\ ,
    \label{eq:chi-W}
  \end{align}
  and combining \eqref{eq:ker-pi-multiple-1} and \eqref{eq:chi-W} we get \eqref{eq:prop:ker-pi}
  \begin{align}
    [M]=[M_n]+(\chi(M_0)-\chi(W))[S^d]=[N]+(\chi(M)-\chi(W))[S^d]\ .
    \label{eq:ker-pi-multiple-2}
  \end{align}

\end{proof}

Now we are ready to present the proof of Theorem~\ref{thm:SKK-ses}.

\begin{proof}[Proof of Theorem~\ref{thm:SKK-ses}]

  \textsl{Part~\ref{part:1}:} This follows directly from Proposition~\ref{prop:ker-pi}.

  \textsl{Part~\ref{part:2}~\ref{part:2a}:}
  Let $\pi:\SKK{d}{G}\twoheadrightarrow\Omega_d^{G'}$ be the map in \eqref{eq:prop:SKK-bord-surj}.
  By Lemma~\ref{lem:eulerchar-SKK} the Euler characteristic is an SKK invariant.
  If $\pi([M])=0$ then $M$ is a boundary of a $(d+1)$-dimensional $G'$-manifold and hence $\chi(M)$ is even.
  We have $\chi(S^{2n})=2$, so altogether we have that $\chi/2:K=\ker(\pi)\to\Zb$ is an isomorphism 
  with inverse $1\mapsto[S^d]$.

  \textsl{Part~\ref{part:2}~\ref{part:2b}:}
  If $\chi(M)$ is even for every closed $d$-dimensional $G$-manifold then $\chi/2$ is a well defined group homomorphism 
  from the whole group $\SKK{d}{G}$ and it is clearly a section of $1\mapsto[S^d]$, so the short exact sequence \eqref{eq:thm:SKK-ses} splits.
  
  \textsl{Part~\ref{part:3}~\ref{part:3a}:}
  Applying Proposition~\ref{prop:ker-pi} for the $G'$-bordism $S^d\xrightarrow{D^{d+1}}\emptyset$ 
  we get that $2[S^d]=0$ in $\SKK{d}{G}$.

  If there exists $\emptyset\xrightarrow{W}\emptyset$ with $\chi(W)$ odd and
  we apply the same proposition to get $\chi(W)[S^d]=0$, and altogether $[S^d]=0$.

  Let us assume now that $\chi(W)$ is even for every closed $(d+1)$-dimensional $G'$-manifold.
  We construct an SKK invariant $\varphi:K\to\Zb/2$ with $\varphi([S^d])=1$ which then proves $K=\Zb/2$.

  Let $[M]\in K$ and choose a $G'$-bordism $M\xrightarrow{U}\emptyset$. We define
  \begin{align}
    \varphi([M]):=\chi(U)\Mod{2}\ .
    \label{eq:def-phi-invariant}
  \end{align}
  By considering $S^d\xrightarrow{D^{d+1}}\emptyset$ we have $\varphi([S^d])=1$.

  We show that $\varphi$ is independent of the choice of the $G'$-bordism $U$:
  Let $M\xrightarrow{V}\emptyset$ be another $G'$-bordism and consider the dual bordism $\emptyset\xrightarrow{V^*}M$.
  We have
  \begin{align}
    \chi(U\circ V^*)=
    \chi(U)+\chi(V^*)-\chi(M)=
    \chi(U)+\chi(V)-\chi(M)\ ,
    \label{eq:phi-indep-of-U}
  \end{align}
  where the left hand side is even by our assumption and $\chi(M)$ is even as $M$ is a boundary,
  so $\chi(U)$ and $\chi(V)$ have the same parity.

  Finally we show that $\varphi$ is an SKK invariant by showing that it takes the same value on both sides of the equivalence relation \eqref{eq:def-SKK-rel}.
  Let $X=M_1\circ N_1+M_2\circ N_2$ and $Y:=M_1\circ N_2+M_2\circ N_1$.
  Pick $Y\xrightarrow{U}\emptyset$, let $X\xrightarrow{W}Y$
  be the bordism in \eqref{eq:SKK-rel-bord} and let $X\xrightarrow{U\circ W}\emptyset$.
  We have $\chi(U\circ W)=\chi(U)+\chi(W)-\chi(Y)$, where the last term is even, as it is a boundary.
  We claim that $\chi(W)$ is even and therefore $\varphi([X])=\varphi([Y])$.
  We compute $\chi(W)$ from the pieces it is constructed from:
  \begin{align}
    \begin{aligned}
    \chi(W)&=
    \chi(X\times[0,1]) 
  -\chi((S+S)\times[ 0,1] ])
    +\chi(W_{\tau})
  -\chi((S+S)\times[ 0,1] ])
    + \chi(Y\times[0,1])\\
    &\equiv
    \chi(X)+\chi( (S+S)\times[0,1]\times[0,1])+\chi(Y)
    \equiv 0 \Mod{2}
    \end{aligned}
    \label{eq:chiW-even}
  \end{align}
  where we used that as a manifold $W_{\tau}\cong (S+S)\times[0,1]\times[0,1]$ and the assumption that $X$ (and $Y$) is a boundary.

  \textsl{Part~\ref{part:3}~\ref{part:3b}:}
This is shown on \cite[Page\,48]{Karras:1973SKKbook}.
  
\end{proof}

\section{Proof of Lemma~\ref{lem:torus-rel}}\label{app:proof}
In this section we will use the combinatorial model of $r$-spin surfaces of
\cite{Novak:2015phd,Runkel:2018rs} to prove Lemma~\ref{lem:torus-rel}. 
We do not wish to present all details of the combinatorial model, we just note that
it consists of a cell decomposition of the surface in question together with a marking,
and refer the reader to \cite[Sec.\,2.3]{Runkel:2018rs}.
A marking consists of an edge orientation and an element in $\Zb/r$ for each edge
and for each face a choice of an edge before gluing the faces along the edges.
There are certain moves between different marked cell decompositions which describe 
isomorphic $r$-spin surfaces, these can be found in \cite[Fig.\,4\,and\,6]{Runkel:2018rs}.

We continue by introducing some notation.
Consider the $r$-spin cylinder $C^\mathrm{in}(\kappa,u)$ 
(resp.\ $C^\mathrm{out}(\kappa,u)$) with two ingoing (resp.\ outgoing) boundary components with boundary type $\kappa$ and $-\kappa$ 
given by the following marked cell decomposition:
\begin{equation}
	\begin{tikzpicture}
	\begin{pgfonlayer}{nodelayer}
		\node [style=black dot small] (15) at (-6.5, -0.75) {};
		\node [style=none] (18) at (-4, -1.25) {};
		\node [style=none] (19) at (-2, -1.25) {};
		\node [style=black dot small] (21) at (-3, -0.75) {};
		\node [style=none] (22) at (-4.75, 0) {$t$};
		\node [style=none] (23) at (-5, -1.25) {$u$};
		\node [style=none] (24) at (-1.5, -1.25) {$0$};
		\node [style=none] (32) at (-3, -2.75) {$S^1_{-\kappa}$};
		\node [style=none] (33) at (-6.5, -2.75) {$S^1_{\kappa}$};
		\node [style=none] (34) at (-3, -1.75) {};
		\node [style=none] (35) at (-7.5, -1.25) {};
		\node [style=none] (36) at (-5.5, -1.25) {};
		\node [style=none] (37) at (-6.5, -1.75) {};
		\node [style=none] (39) at (8.5, 1.25) {};
		\node [style=none] (42) at (7.75, -0.25) {$t$};
		\node [style=none] (43) at (7.5, 1.25) {$u$};
		\node [style=none] (44) at (11, 1.25) {$0$};
		\node [style=none] (45) at (9.5, 2.75) {$S^1_{-\kappa}$};
		\node [style=none] (46) at (6, 2.75) {$S^1_{\kappa}$};
		\node [style=none] (47) at (9.5, 1.75) {};
		\node [style=none] (48) at (5, 1.25) {};
		\node [style=none] (50) at (6, 1.75) {};
		\node [style=none] (51) at (7, 1.25) {};
		\node [style=black dot small] (52) at (6, 0.75) {};
		\node [style=black dot small] (53) at (9.5, 0.75) {};
		\node [style=none] (54) at (10.5, 1.25) {};
		\node [style=none] (55) at (-10, 0) {$C^\mathrm{in}(\kappa,u)=$};
		\node [style=none] (56) at (2.5, 0) {$C^\mathrm{out}(\kappa,u)=$};
		\node [style=none] (57) at (-4.75, 2.75) {$\emptyset$};
		\node [style=none] (58) at (7.75, -2.75) {$\emptyset$};
	\end{pgfonlayer}
	\begin{pgfonlayer}{edgelayer}
		\draw [style=mid arrow, in=90, out=90, looseness=1.50] (21) to (15);
		\draw [style=mid arrow, in=180, out=90, looseness=0.75] (18.center) to (21);
		\draw [style=marked edge, in=90, out=0, looseness=0.75] (21) to (19.center);
		\draw [in=180, out=-90, looseness=0.75] (18.center) to (34.center);
		\draw [in=-90, out=0, looseness=0.75] (34.center) to (19.center);
		\draw [in=180, out=-90, looseness=0.75] (35.center) to (37.center);
		\draw [in=-90, out=0, looseness=0.75] (37.center) to (36.center);
		\draw [in=0, out=90, looseness=0.75] (36.center) to (15);
		\draw [style=mid arrow, in=180, out=90, looseness=0.75] (35.center) to (15);
		\draw [in=90, out=90, looseness=1.75] (35.center) to (19.center);
		\draw [in=90, out=90, looseness=1.75] (36.center) to (18.center);
		\draw [in=-180, out=90, looseness=0.75] (39.center) to (47.center);
		\draw [in=-180, out=90, looseness=0.75] (48.center) to (50.center);
		\draw [style=marked edge, in=0, out=-90, looseness=0.75] (51.center) to (52);
		\draw [in=0, out=-90, looseness=0.75] (54.center) to (53);
		\draw [style=mid arrow, in=180, out=-90, looseness=0.75] (39.center) to (53);
		\draw [in=90, out=0, looseness=0.75] (47.center) to (54.center);
		\draw [in=90, out=0, looseness=0.75] (50.center) to (51.center);
		\draw [style=mid arrow, in=-180, out=-90, looseness=0.75] (48.center) to (52);
		\draw [in=-90, out=-90, looseness=1.75] (48.center) to (54.center);
		\draw [in=-90, out=-90, looseness=1.75] (51.center) to (39.center);
		\draw [style=mid arrow, in=270, out=-90, looseness=1.50] (52) to (53);
	\end{pgfonlayer}
\end{tikzpicture}
\ .
	\label{eq:cylinders}
\end{equation}
Let 
\begin{align}
	T(\kappa,u):=C^\mathrm{in}(\kappa,0)\circ C^\mathrm{out}(\kappa,u)
	\label{eq:def-r-spin-torus}
\end{align}
be the $r$-spin torus obtained from the composition of the cylinders in \eqref{eq:cylinders},
see also \cite[Sect.\,3.1.2]{Stern:2020oc} for the composition of $r$-spin bordisms.

\begin{proof}[Proof of Lemma~\ref{lem:torus-rel}]
	We start by proving \eqref{eq:lem:torus-rel:1}.
	Let us consider the equation in $\SKK{2}{\spin{r}}$:
	\begin{align}
		\input{torus-rel-1.tikz}
		\label{eq:pf:lem:torus-rel:1}
	\end{align}
	The first term on the right hand side 
	of \eqref{eq:pf:lem:torus-rel:1} is
	\begin{align}
		\begin{tikzpicture}
	\begin{pgfonlayer}{nodelayer}
		\node [style=none] (0) at (-10.75, 13) {$-\kappa$};
		\node [style=none] (7) at (-3.75, 13) {$-\kappa$};
		\node [style=none] (9) at (-7.25, 10) {$\kappa$};
		\node [style=none] (11) at (-7.25, 6.5) {$-\kappa$};
		\node [style=none] (13) at (-12.5, 11) {};
		\node [style=black dot small] (36) at (-12.5, 12) {};
		\node [style=black dot small] (40) at (-9, 12) {};
		\node [style=black dot small] (41) at (-2, 12) {};
		\node [style=black dot small] (45) at (-5.5, 12) {};
		\node [style=none] (53) at (-13.5, 11.5) {};
		\node [style=none] (54) at (-11.5, 11.5) {};
		\node [style=none] (68) at (-4, 11.5) {$0$};
		\node [style=none] (69) at (-7.5, 11.5) {$u_2$};
		\node [style=none] (70) at (-0.5, 11.5) {$0$};
		\node [style=none] (71) at (-11, 11.5) {$u_1$};
		\node [style=none] (72) at (-10, 11.5) {};
		\node [style=none] (73) at (-9, 11) {};
		\node [style=none] (74) at (-8, 11.5) {};
		\node [style=none] (75) at (-6.5, 11.5) {};
		\node [style=none] (76) at (-5.5, 11) {};
		\node [style=none] (77) at (-4.5, 11.5) {};
		\node [style=none] (78) at (-3, 11.5) {};
		\node [style=none] (79) at (-2, 11) {};
		\node [style=none] (80) at (-1, 11.5) {};
		\node [style=none] (81) at (3.25, 13) {$-\kappa$};
		\node [style=none] (82) at (10.25, 13) {$-\kappa$};
		\node [style=none] (83) at (6.75, 9.25) {$-\kappa-1$};
		\node [style=none] (84) at (6.75, 6.5) {$-\kappa$};
		\node [style=none] (85) at (1.5, 11) {};
		\node [style=black dot small] (87) at (5, 12) {};
		\node [style=black dot small] (88) at (12, 12) {};
		\node [style=black dot small] (89) at (8.5, 12) {};
		\node [style=none] (90) at (0.5, 11.5) {};
		\node [style=none] (92) at (10, 11.5) {$-1$};
		\node [style=none] (93) at (6.5, 11.5) {$u_2$};
		\node [style=none] (94) at (13.5, 11.5) {$0$};
		\node [style=none] (95) at (3, 11.5) {$u_1$};
		\node [style=none] (96) at (4, 11.5) {};
		\node [style=none] (97) at (5, 11) {};
		\node [style=none] (98) at (6, 11.5) {};
		\node [style=none] (99) at (7.5, 11.5) {};
		\node [style=none] (100) at (8.5, 11) {};
		\node [style=none] (101) at (9.5, 11.5) {};
		\node [style=none] (102) at (11, 11.5) {};
		\node [style=none] (103) at (12, 11) {};
		\node [style=none] (104) at (13, 11.5) {};
		\node [style=none] (105) at (-0.25, 9.5) {$=$};
		\node [style=black dot small] (106) at (1.5, 12) {};
		\node [style=none] (107) at (2.5, 11.5) {};
		\node [style=none] (108) at (-10.75, 3.5) {$-\kappa$};
		\node [style=none] (109) at (-3.75, 3.5) {$-\kappa$};
		\node [style=none] (110) at (-7.25, -0.25) {$-\kappa$};
		\node [style=none] (111) at (-7.25, -3) {$-\kappa$};
		\node [style=none] (112) at (-12.5, 1.5) {};
		\node [style=black dot small] (113) at (-9, 2.5) {};
		\node [style=black dot small] (114) at (-2, 2.5) {};
		\node [style=none] (116) at (-13.5, 2) {};
		\node [style=none] (117) at (-4, 2) {$0$};
		\node [style=none] (118) at (-7.5, 2) {$u_2$};
		\node [style=none] (119) at (-0.5, 2) {$0$};
		\node [style=none] (120) at (-11, 2) {$u_1$};
		\node [style=none] (121) at (-10, 2) {};
		\node [style=none] (122) at (-9, 1.5) {};
		\node [style=none] (123) at (-8, 2) {};
		\node [style=none] (125) at (-5.5, 1.5) {};
		\node [style=none] (127) at (-3, 2) {};
		\node [style=none] (128) at (-2, 1.5) {};
		\node [style=none] (130) at (-14.25, 0) {$=$};
		\node [style=black dot small] (131) at (-12.5, 2.5) {};
		\node [style=none] (132) at (-11.5, 2) {};
		\node [style=black dot small] (135) at (-5.5, 2.5) {};
		\node [style=none] (136) at (-4.5, 2) {};
		\node [style=none] (137) at (-6.5, 2) {};
		\node [style=none] (138) at (-1, 2) {};
		\node [style=none] (139) at (-0.25, 0) {$=$};
		\node [style=none] (140) at (3.25, 2.5) {$-\kappa$};
		\node [style=none] (141) at (1.5, 0.5) {};
		\node [style=none] (144) at (6.5, 1) {$u_2$};
		\node [style=none] (145) at (3, 1) {$u_1$};
		\node [style=none] (147) at (5, 0.5) {};
		\node [style=none] (151) at (3.25, -1.25) {$-\kappa$};
		\node [style=none] (153) at (0.5, 1) {};
		\node [style=none] (155) at (2.5, 1) {};
		\node [style=none] (159) at (6, 1) {};
		\node [style=none] (163) at (4, 1) {};
		\node [style=none] (165) at (6.75, 0) {$=$};
		\node [style=none] (166) at (10.25, 2.5) {$-\kappa$};
		\node [style=none] (167) at (8.5, 0.5) {};
		\node [style=none] (168) at (13.5, 1) {$0$};
		\node [style=none] (169) at (10, 1) {$v$};
		\node [style=none] (170) at (12, 0.5) {};
		\node [style=none] (171) at (10.25, -1.25) {$-\kappa$};
		\node [style=none] (173) at (7.5, 1) {};
		\node [style=none] (174) at (9.5, 1) {};
		\node [style=none] (175) at (13, 1) {};
		\node [style=none] (177) at (11, 1) {};
		\node [style=black dot small] (180) at (1.5, 1.5) {};
		\node [style=black dot small] (181) at (5, 1.5) {};
		\node [style=black dot small] (182) at (8.5, 1.5) {};
		\node [style=black dot small] (183) at (12, 1.5) {};
	\end{pgfonlayer}
	\begin{pgfonlayer}{edgelayer}
		\draw [style=thick dashed black line, in=-180, out=-90, looseness=0.75] (72.center) to (73.center);
		\draw [style=thick dashed black line, in=-90, out=0, looseness=0.75] (73.center) to (74.center);
		\draw [style=thick dashed black line, in=-180, out=-90, looseness=0.75] (75.center) to (76.center);
		\draw [style=thick dashed black line, in=-90, out=0, looseness=0.75] (76.center) to (77.center);
		\draw [style=thick dashed black line, in=-180, out=-90, looseness=0.75] (78.center) to (79.center);
		\draw [style=thick dashed black line, in=-90, out=0, looseness=0.75] (79.center) to (80.center);
		\draw [style=mid arrow, in=90, out=90, looseness=1.50] (41) to (45);
		\draw [style=mid arrow, in=180, out=90, looseness=0.75] (78.center) to (41);
		\draw [style=mid arrow, in=180, out=90, looseness=0.75] (75.center) to (45);
		\draw [in=90, out=90, looseness=1.75] (75.center) to (80.center);
		\draw [in=90, out=90, looseness=1.75] (77.center) to (78.center);
		\draw [style=mid arrow, in=90, out=90, looseness=1.50] (40) to (36);
		\draw [style=mid arrow, in=180, out=90, looseness=0.75] (72.center) to (40);
		\draw [style=marked edge, in=90, out=0, looseness=0.75] (40) to (74.center);
		\draw [style=marked edge, in=90, out=0, looseness=0.75] (41) to (80.center);
		\draw [style=thick line, in=0, out=90, looseness=0.75] (77.center) to (45);
		\draw [style=mid arrow, in=-90, out=-90, looseness=1.50] (36) to (41);
		\draw [style=mid arrow, in=270, out=-90, looseness=2.25] (40) to (45);
		\draw [in=-90, out=-90, looseness=1.75] (77.center) to (72.center);
		\draw [in=-90, out=-90, looseness=1.50] (80.center) to (53.center);
		\draw [in=-90, out=-90, looseness=1.75] (75.center) to (74.center);
		\draw [in=-90, out=-90, looseness=1.50] (78.center) to (54.center);
		\draw [style=mid arrow, in=180, out=90, looseness=0.75] (53.center) to (36);
		\draw [style=thick dashed black line, in=-180, out=-90, looseness=0.75] (53.center) to (13.center);
		\draw [style=thick dashed black line, in=-90, out=0, looseness=0.75] (13.center) to (54.center);
		\draw [in=90, out=90, looseness=1.75] (53.center) to (74.center);
		\draw [in=90, out=90, looseness=1.75] (54.center) to (72.center);
		\draw [style=thick dashed black line, in=-180, out=-90, looseness=0.75] (96.center) to (97.center);
		\draw [style=thick dashed black line, in=-90, out=0, looseness=0.75] (97.center) to (98.center);
		\draw [style=thick dashed black line, in=-180, out=-90, looseness=0.75] (99.center) to (100.center);
		\draw [style=thick dashed black line, in=-90, out=0, looseness=0.75] (100.center) to (101.center);
		\draw [style=thick dashed black line, in=-180, out=-90, looseness=0.75] (102.center) to (103.center);
		\draw [style=thick dashed black line, in=-90, out=0, looseness=0.75] (103.center) to (104.center);
		\draw [style=mid arrow, in=90, out=90, looseness=1.50] (88) to (89);
		\draw [style=mid arrow, in=180, out=90, looseness=0.75] (102.center) to (88);
		\draw [in=90, out=90, looseness=1.75] (99.center) to (104.center);
		\draw [in=90, out=90, looseness=1.75] (101.center) to (102.center);
		\draw [style=mid arrow, in=180, out=90, looseness=0.75] (96.center) to (87);
		\draw [style=marked edge, in=90, out=0, looseness=0.75] (87) to (98.center);
		\draw [style=marked edge, in=90, out=0, looseness=0.75] (88) to (104.center);
		\draw [in=-90, out=-90, looseness=1.75] (101.center) to (96.center);
		\draw [in=-90, out=-90, looseness=1.50] (104.center) to (90.center);
		\draw [in=-90, out=-90, looseness=1.75] (99.center) to (98.center);
		\draw [style=thick dashed black line, in=-180, out=-90, looseness=0.75] (90.center) to (85.center);
		\draw [in=90, out=90, looseness=1.75] (90.center) to (98.center);
		\draw [style=mid arrow, in=270, out=-90, looseness=2.25] (89) to (87);
		\draw [style=marked edge, in=0, out=90, looseness=0.75] (54.center) to (36);
		\draw [style=marked edge, in=0, out=90, looseness=0.75] (107.center) to (106);
		\draw [style=mid arrow, in=90, out=90, looseness=1.50] (87) to (106);
		\draw [style=mid arrow, in=-90, out=-90, looseness=1.50] (106) to (88);
		\draw [in=-90, out=-90, looseness=1.50] (102.center) to (107.center);
		\draw [style=mid arrow, in=180, out=90, looseness=0.75] (90.center) to (106);
		\draw [style=thick dashed black line, in=-90, out=0, looseness=0.75] (85.center) to (107.center);
		\draw [in=90, out=90, looseness=1.75] (107.center) to (96.center);
		\draw [style=marked edge, in=0, out=90, looseness=0.75] (74.center) to (40);
		\draw [style=marked edge, in=360, out=90, looseness=0.75] (98.center) to (87);
		\draw [style=thick dashed black line, in=-180, out=-90, looseness=0.75] (121.center) to (122.center);
		\draw [style=thick dashed black line, in=-90, out=0, looseness=0.75] (122.center) to (123.center);
		\draw [style=thick dashed black line, in=-180, out=-90, looseness=0.75] (127.center) to (128.center);
		\draw [style=mid arrow, in=180, out=90, looseness=0.75] (127.center) to (114);
		\draw [style=mid arrow, in=180, out=90, looseness=0.75] (121.center) to (113);
		\draw [style=marked edge, in=90, out=0, looseness=0.75] (113) to (123.center);
		\draw [style=thick dashed black line, in=-180, out=-90, looseness=0.75] (116.center) to (112.center);
		\draw [in=90, out=90, looseness=1.75] (116.center) to (123.center);
		\draw [style=marked edge, in=0, out=90, looseness=0.75] (132.center) to (131);
		\draw [style=mid arrow, in=90, out=90, looseness=1.50] (113) to (131);
		\draw [style=mid arrow, in=-90, out=-90, looseness=1.50] (131) to (114);
		\draw [in=-90, out=-90, looseness=1.50] (127.center) to (132.center);
		\draw [style=mid arrow, in=180, out=90, looseness=0.75] (116.center) to (131);
		\draw [style=thick dashed black line, in=-90, out=0, looseness=0.75] (112.center) to (132.center);
		\draw [in=90, out=90, looseness=1.75] (132.center) to (121.center);
		\draw [in=90, out=90, looseness=1.75] (137.center) to (138.center);
		\draw [style=marked edge, in=0, out=90, looseness=0.75] (136.center) to (135);
		\draw [style=mid arrow, in=90, out=-180, looseness=0.75] (89) to (99.center);
		\draw [style=thick line, in=0, out=90, looseness=0.75] (101.center) to (89);
		\draw [style=mid arrow, in=90, out=180, looseness=0.75] (135) to (137.center);
		\draw [style=thick dashed black line, in=-180, out=-90, looseness=0.75] (137.center) to (125.center);
		\draw [style=thick dashed black line, in=-90, out=0, looseness=0.75] (125.center) to (136.center);
		\draw [style=thick dashed black line, in=-90, out=0, looseness=0.75] (128.center) to (138.center);
		\draw [style=mid arrow, in=90, out=90, looseness=1.50] (114) to (135);
		\draw [in=90, out=90, looseness=1.75] (136.center) to (127.center);
		\draw [style=marked edge, in=90, out=0, looseness=0.75] (114) to (138.center);
		\draw [in=-90, out=-90, looseness=1.75] (136.center) to (121.center);
		\draw [in=-90, out=-90, looseness=1.50] (138.center) to (116.center);
		\draw [in=-90, out=-90, looseness=1.75] (137.center) to (123.center);
		\draw [style=mid arrow, in=270, out=-90, looseness=2.25] (135) to (113);
		\draw [style=thick dashed black line, in=-180, out=-90, looseness=0.75] (163.center) to (147.center);
		\draw [style=thick dashed black line, in=-90, out=0, looseness=0.75] (147.center) to (159.center);
		\draw [style=thick dashed black line, in=-180, out=-90, looseness=0.75] (153.center) to (141.center);
		\draw [in=90, out=90, looseness=1.75] (153.center) to (159.center);
		\draw [style=thick dashed black line, in=-90, out=0, looseness=0.75] (141.center) to (155.center);
		\draw [in=90, out=90, looseness=1.75] (155.center) to (163.center);
		\draw [in=-90, out=-90, looseness=1.75] (159.center) to (153.center);
		\draw [in=-90, out=-90, looseness=1.75] (163.center) to (155.center);
		\draw [style=thick dashed black line, in=-180, out=-90, looseness=0.75] (177.center) to (170.center);
		\draw [style=thick dashed black line, in=-90, out=0, looseness=0.75] (170.center) to (175.center);
		\draw [style=thick dashed black line, in=-180, out=-90, looseness=0.75] (173.center) to (167.center);
		\draw [in=90, out=90, looseness=1.75] (173.center) to (175.center);
		\draw [style=thick dashed black line, in=-90, out=0, looseness=0.75] (167.center) to (174.center);
		\draw [in=90, out=90, looseness=1.75] (174.center) to (177.center);
		\draw [in=-90, out=-90, looseness=1.75] (175.center) to (173.center);
		\draw [in=-90, out=-90, looseness=1.75] (177.center) to (174.center);
		\draw [style=mid arrow, in=270, out=-90, looseness=2.25] (180) to (181);
		\draw [style=mid arrow, in=270, out=-90, looseness=2.25] (182) to (183);
		\draw [style=mid arrow, in=180, out=90, looseness=0.75] (163.center) to (181);
		\draw [style=marked edge, in=90, out=0, looseness=0.75] (181) to (159.center);
		\draw [style=marked edge, in=0, out=90, looseness=0.75] (155.center) to (180);
		\draw [style=mid arrow, in=90, out=90, looseness=1.50] (181) to (180);
		\draw [style=mid arrow, in=180, out=90, looseness=0.75] (153.center) to (180);
		\draw [style=mid arrow, in=180, out=90, looseness=0.75] (177.center) to (183);
		\draw [style=marked edge, in=90, out=0, looseness=0.75] (183) to (175.center);
		\draw [style=marked edge, in=0, out=90, looseness=0.75] (174.center) to (182);
		\draw [style=mid arrow, in=90, out=90, looseness=1.50] (183) to (182);
		\draw [style=mid arrow, in=180, out=90, looseness=0.75] (173.center) to (182);
	\end{pgfonlayer}
\end{tikzpicture}

		\label{eq:pf:lem:torus-rel:1-a}
	\end{align}
	with $v=u_1-u_2$,
	where we used the moves of \cite[Fig.\,4\,and\,6]{Runkel:2018rs}.
	The second term on the right hand side 
	of \eqref{eq:pf:lem:torus-rel:1} is
	\begin{align}
		\begin{tikzpicture}
	\begin{pgfonlayer}{nodelayer}
		\node [style=none] (0) at (-0.75, -3) {$-\kappa$};
		\node [style=none] (1) at (-7.75, -3) {$-\kappa$};
		\node [style=black dot small] (3) at (-2.5, -1) {};
		\node [style=none] (4) at (-4.25, 0) {$\kappa$};
		\node [style=black dot small] (7) at (1, -1) {};
		\node [style=none] (8) at (-4.25, 3.5) {$-\kappa$};
		\node [style=black dot small] (31) at (-9.5, -1) {};
		\node [style=black dot small] (32) at (-6, -1) {};
		\node [style=none] (42) at (-1, -1.5) {$u_4$};
		\node [style=none] (43) at (2.5, -1.5) {$0$};
		\node [style=none] (44) at (-8, -1.5) {$u_3$};
		\node [style=none] (45) at (-4.5, -1.5) {$0$};
		\node [style=none] (48) at (-9.5, -2) {};
		\node [style=none] (49) at (-10.5, -1.5) {};
		\node [style=none] (50) at (-8.5, -1.5) {};
		\node [style=none] (51) at (-7, -1.5) {};
		\node [style=none] (52) at (-6, -2) {};
		\node [style=none] (53) at (-5, -1.5) {};
		\node [style=none] (54) at (-3.5, -1.5) {};
		\node [style=none] (55) at (-2.5, -2) {};
		\node [style=none] (56) at (-1.5, -1.5) {};
		\node [style=none] (57) at (0, -1.5) {};
		\node [style=none] (58) at (1, -2) {};
		\node [style=none] (59) at (2, -1.5) {};
		\node [style=none] (62) at (3.75, 0) {$=\dots=$};
		\node [style=none] (63) at (8.25, 1) {$-\kappa$};
		\node [style=none] (64) at (6.5, -1) {};
		\node [style=none] (65) at (11.5, -0.5) {$0$};
		\node [style=none] (66) at (8, -0.5) {$w$};
		\node [style=none] (67) at (10, -1) {};
		\node [style=none] (68) at (8.25, -2.75) {$-\kappa$};
		\node [style=none] (69) at (5.5, -0.5) {};
		\node [style=none] (70) at (7.5, -0.5) {};
		\node [style=none] (71) at (11, -0.5) {};
		\node [style=none] (72) at (9, -0.5) {};
		\node [style=black dot small] (73) at (6.5, 0) {};
		\node [style=black dot small] (74) at (10, 0) {};
	\end{pgfonlayer}
	\begin{pgfonlayer}{edgelayer}
		\draw [style=mid arrow, in=90, out=90, looseness=1.50] (3) to (32);
		\draw [style=mid arrow, in=90, out=90, looseness=1.25] (7) to (31);
		\draw [style=thick dashed black line, in=-180, out=-90, looseness=0.75] (51.center) to (52.center);
		\draw [style=thick dashed black line, in=-90, out=0, looseness=0.75] (52.center) to (53.center);
		\draw [style=thick dashed black line, in=-180, out=-90, looseness=0.75] (54.center) to (55.center);
		\draw [style=thick dashed black line, in=-90, out=0, looseness=0.75] (55.center) to (56.center);
		\draw [style=thick dashed black line, in=-180, out=-90, looseness=0.75] (57.center) to (58.center);
		\draw [style=thick dashed black line, in=-90, out=0, looseness=0.75] (58.center) to (59.center);
		\draw [style=thick dashed black line, in=-180, out=-90, looseness=0.75] (49.center) to (48.center);
		\draw [style=thick dashed black line, in=-90, out=0, looseness=0.75] (48.center) to (50.center);
		\draw [style=mid arrow, in=270, out=-90, looseness=2.50] (31) to (32);
		\draw [style=mid arrow, in=-90, out=-90, looseness=2.50] (3) to (7);
		\draw [style=marked edge, in=360, out=90, looseness=0.75] (56.center) to (3);
		\draw [style=marked edge, in=360, out=90, looseness=0.75] (50.center) to (31);
		\draw [style=thick dashed black line, in=-180, out=-90, looseness=0.75] (72.center) to (67.center);
		\draw [style=thick dashed black line, in=-90, out=0, looseness=0.75] (67.center) to (71.center);
		\draw [style=thick dashed black line, in=-180, out=-90, looseness=0.75] (69.center) to (64.center);
		\draw [in=90, out=90, looseness=1.75] (69.center) to (71.center);
		\draw [style=thick dashed black line, in=-90, out=0, looseness=0.75] (64.center) to (70.center);
		\draw [in=90, out=90, looseness=1.75] (70.center) to (72.center);
		\draw [in=-90, out=-90, looseness=1.75] (71.center) to (69.center);
		\draw [in=-90, out=-90, looseness=1.75] (72.center) to (70.center);
		\draw [style=mid arrow, in=270, out=-90, looseness=2.25] (73) to (74);
		\draw [style=mid arrow, in=180, out=90, looseness=0.75] (72.center) to (74);
		\draw [style=marked edge, in=90, out=0, looseness=0.75] (74) to (71.center);
		\draw [style=marked edge, in=0, out=90, looseness=0.75] (70.center) to (73);
		\draw [style=mid arrow, in=90, out=90, looseness=1.50] (74) to (73);
		\draw [style=mid arrow, in=180, out=90, looseness=0.75] (69.center) to (73);
		\draw [in=-90, out=-90, looseness=1.75] (50.center) to (51.center);
		\draw [style=mid arrow, in=180, out=90, looseness=0.75] (54.center) to (3);
		\draw [style=marked edge, in=90, out=0, looseness=0.75] (3) to (56.center);
		\draw [in=90, out=90, looseness=1.75] (51.center) to (56.center);
		\draw [style=mid arrow, in=180, out=90, looseness=0.75] (57.center) to (7);
		\draw [style=marked edge, in=90, out=0, looseness=0.75] (7) to (59.center);
		\draw [in=90, out=90, looseness=1.50] (49.center) to (59.center);
		\draw [in=90, out=90, looseness=1.75] (53.center) to (54.center);
		\draw [in=90, out=90, looseness=1.50] (50.center) to (57.center);
		\draw [style=thick line, in=0, out=90, looseness=0.75] (53.center) to (32);
		\draw [style=mid arrow, in=180, out=90, looseness=0.75] (51.center) to (32);
		\draw [style=mid arrow, in=180, out=90, looseness=0.75] (49.center) to (31);
		\draw [in=-90, out=-90, looseness=1.75] (49.center) to (53.center);
		\draw [in=-90, out=-90, looseness=1.75] (54.center) to (59.center);
		\draw [in=-90, out=-90, looseness=1.75] (56.center) to (57.center);
	\end{pgfonlayer}
\end{tikzpicture}

		\label{eq:pf:lem:torus-rel:1-b}
	\end{align}
	with $w=u_3+u_4$. Finally, using the $SL(2,\Zb)$ action on 
	the set of isomorphism classes of $r$-spin structures on the torus,
	we have $[T(-\kappa,u_2)]=[T(\kappa,-u_2)]$. Putting all the above together
	we obtain \eqref{eq:lem:torus-rel:1}.

	\medskip
	
	We continue with proving \eqref{eq:lem:torus-rel:2}.
	Pick arbitrary $r$-spin structures on $\Sigma_{1,1}^\mathrm{in}$ and
	on $\Sigma_{1,1}^\mathrm{out}$.
	Now consider the equation in $\SKK{2}{\spin{r}}$:
	\begin{align}
		\begin{aligned}
		&\left( \Sigma_{0,1}^\mathrm{in}\sqcup \Sigma_{1,1}^\mathrm{in}\sqcup C^\mathrm{in}(1,0) \right)\circ
		\left( \Sigma_{1,1}^\mathrm{out}\sqcup \Sigma_{0,1}^\mathrm{out}\sqcup C^\mathrm{out}(1,u) \right)\\
		=&
		\left( C^\mathrm{in}(1,0) \sqcup \Sigma_{0,1}^\mathrm{in}\sqcup \Sigma_{1,1}^\mathrm{in} \right)\circ
		\left( C^\mathrm{out}(1,u) \sqcup \Sigma_{1,1}^\mathrm{out}\sqcup \Sigma_{0,1}^\mathrm{out} \right)\ ,
		\end{aligned}
		\label{eq:pf:lem:torus-rel:2-a}
	\end{align}
	where we cut along $S^1_{-1}\sqcup S^1_{1}\sqcup S^1_{-1}\sqcup S^1_{1}$.
	Notice, since up to isomorphism there is a unique $r$-spin structure on the disc
	the two leftmost $r$-spin tori of the left hand side are
	equal to the two $r$-spin tori on the right hand side, giving \eqref{eq:lem:torus-rel:2} directly.

	Finally we compute the $k$-invariant of the $r$-spin circle 
	$[S^1_{\kappa}]$ for $\kappa\in\Zb/r$. The duality morphisms for $S^1_{\kappa}$
	are $C^\mathrm{in}(\kappa,0)$ and $C^\mathrm{out}(\kappa,0)$ from 
	\eqref{eq:cylinders}.
	The composition 
	\begin{align}
		k([S^1_{\kappa}])=
		[C^\mathrm{in}(\kappa,0)\circ\sigma_{S^1_{-\kappa},S^1_{\kappa}}
		\circ C^\mathrm{out}(\kappa,0)]
		\label{eq:k-inv-circle}
	\end{align}
	is an $r$-spin torus $[T(\kappa,0)]$ after applying a Dehn-twist.

\end{proof}

\phantomsection

\end{document}